\documentclass[12pt]{amsart}

\usepackage{amsmath,amssymb,amscd,amsxtra,upref}
\usepackage{latexsym}
\usepackage[dvips]{graphics,epsfig}
\usepackage{enumerate}

\usepackage[colorlinks=true, pdfstartview=FitV, linkcolor=blue, citecolor=blue, urlcolor=blue]{hyperref}

\newtheorem{theor}{Theorem}
\newtheorem{lemma}[theor]{Lemma}
\newtheorem{prop}[theor]{Proposition}
\newtheorem{coro}[theor]{Corollary}

\newtheorem{thm}{Theorem} 

\theoremstyle{remark}

\newtheorem{remark}{Remark}[section]

\newcommand{\na}{\mathbb{N}}
\newcommand{\re}{\mathbb{R}}
\newcommand{\rn}{\mathbb{R}^n}
\newcommand{\ent}{\mathbb{Z}}
\newcommand{\abs}[1]{\vert #1 \vert}
\newcommand{\norm}[2]{\|#1\|_{#2}}

\newcommand{\Sw}{\mathcal{S}(\rn)}

\begin{document}



\thanks{First author is supported by the ANR under the project AFoMEN no. 2011-JS01-001-01. Second, third, and fourth authors  supported by the NSF under grants DMS 0901587, DMS 1201504, and DMS 1101327, respectively.}

\address{Fr\'ed\'eric Bernicot, CNRS-Universit\'e de Nantes, Laboratoire Jean Leray. 2, rue de la Houssini\`ere
44322 Nantes cedex 3 (France).} \email{frederic.bernicot@univ-nantes.fr }

\address{Diego Maldonado, Department of Mathematics, Kansas State University. 138 Cardwell Hall,
Manhattan, KS-66506 (USA).} \email{dmaldona@math.ksu.edu}

\address{Kabe Moen, Department of Mathematics, University of Alabama, Tuscaloosa, AL-35487-0350
(USA).}\email{kmoen@as.ua.edu}

\address{Virginia Naibo, Department of Mathematics, Kansas State University. 138 Cardwell Hall,
Manhattan, KS-66506 (USA).} \email{vnaibo@math.ksu.edu}

\title{Bilinear Sobolev-Poincar\'e inequalities and Leibniz-type rules}

\author{Fr\'ed\'eric Bernicot, Diego Maldonado, Kabe Moen, and Virginia Naibo}
\date{\today}

\begin{abstract}
The dual purpose of this article is to establish bilinear Poincar\'e-type estimates associated to an approximation of the identity and to explore the connections between bilinear pseudo-differential operators and bilinear potential-type operators. The common underlying theme in both topics is their applications to Leibniz-type rules in Sobolev and Campanato-Morrey spaces under Sobolev scaling.
\end{abstract}

\maketitle

\section{Introduction}\label{sec:intro}

Leibniz-type rules quantify the regularity of a product of functions in terms of the regularity of its factors. In this sense, Leibniz-type rules are represented by inequalities of the form
\begin{equation}\label{leibrule}
\norm{fg}{Z} \lesssim \norm{f}{X_1} \norm{g}{Y_1} + \norm{f}{X_2}\norm{g}{Y_2},
\end{equation}
where $X_1$, $X_2$, $Y_1$, $Y_2$, and $Z$ are appropriate functional spaces. Along these lines, perhaps the better-known Leibniz-type rules correspond to the fractional Leibniz rules, pioneered by Kato-Ponce \cite{KaPo},  Christ-Weinstein \cite{CWein} and Kenig-Ponce-Vega \cite{KPVe} in their work on PDEs, where the spaces $X_1$, $X_2$, $Y_1$, $Y_2$, and $Z$ belong to the scale of Sobolev spaces $W^{m,p}$; namely,
\begin{equation}\label{katoponce}
\norm{fg}{W^{m,q}} \lesssim \norm{f}{W^{m,p_1}} \norm{g}{L^{p_2}} + \norm{f}{L^{p_1}}\norm{g}{W^{m,p_2}},
\end{equation}
where $m \geq 0$ and
\begin{equation}\label{Holderscale}
\frac{1}{q} = \frac{1}{p_1} + \frac{1}{p_2} \quad  \text{ with }  1 < p_1, p_2 < \infty, 1 \leq q.
\end{equation}
The estimates \eqref{katoponce} follow as a consequence of interpolation and the boundedness properties on products of  Lebesgue spaces of bilinear Coifman-Meyer multipliers (\cite{cm,GT}): If $\sigma$ satisfies
\begin{equation}\label{cmmultipliers}
|\partial_\xi^\alpha \partial_\eta^\beta \sigma(\xi, \eta)|\leq C_{\alpha,
\beta} (|\xi|+|\eta|)^{ - (|\alpha|+|\beta|)}, \quad \xi, \eta
\in \re^n,\,\alpha, \,\beta \in \na_0^n,\,|\alpha|+|\beta|\le C_n,
\end{equation}
where $C_n$ is a certain constant depending only on $n,$
and
\[
T_\sigma(f,g)(x) := \int_{\re^{2n}} \sigma(\xi,\eta) \hat{f}(\xi) \hat{g}(\eta) e^{i x \cdot
(\xi + \eta)} \, d\xi \, d\eta, \qquad x\in\re^n, f, g \in \Sw,
\]
then $T_\sigma$ is bounded from $L^{p_1}\times L^{p_2}$ into $L^q$, where $p_1, p_2,$ and $q$ conform to the H\"older scaling \eqref{Holderscale}.
Then, inequalities \eqref{katoponce} are obtained from this result after observing that, by frequency decoupling, the identity
 \begin{equation}\label{freqdecoup}
 J^m (fg)(x)= T_{\sigma_1}(J^m f, g)(x)+ T_{\sigma_2}(f, J^m g)(x),
 \end{equation}
holds true for some bilinear symbols $\sigma_1$ and $\sigma_2$  and
$$
\widehat{J^m(h)}(\xi):=(1+\abs{\xi}^2)^{m/2}\,\hat{h}(\xi), \quad \xi\in\re^n, m > 0, h \in \Sw.
$$
For $m$ sufficiently large, depending only on dimension, the symbols $\sigma_1$ and $\sigma_2$ satisfy  \eqref{cmmultipliers} and interpolation with the case $m=0$ (notice that here $q \geq 1$) gives \eqref{katoponce} for any $m>0$.
Two immediate conclusions can be derived from this approach. First, since the symbol $\sigma_0 \equiv 1$ satisfies \eqref{cmmultipliers} and yields, through $T_{\sigma_0}$, the product of two functions, the H\"older scaling \eqref{Holderscale} occurs naturally. Second, the identity \eqref{freqdecoup} can be exploited to produce Leibniz-type rules \eqref{leibrule} involving function spaces that interact well with  $J^m$ (for example, Besov and Triebel-Lizorkin spaces) provided that mapping properties for bilinear multipliers $T_\sigma$ are established for such spaces. Indeed, implementations of this program (see, for instance, \cite{BerT, GK, MN09b}), produce Besov, Triebel-Lizorkin, and mixed Besov-Lebesgue Leibniz-type rules.

 A Littlewood-Paley-free path towards Leibniz-type rules was introduced in \cite{MMN} in
 the scales of Campanato-Morrey spaces. In this context, the role of
the identity \eqref{freqdecoup} is played by the inequality
\begin{equation}\label{repformbilinear}
|f(x)g(x) - f_B g_B| \lesssim \mathcal{I}_1(|\nabla f| \chi_B, |g| \chi_B) + \mathcal{I}_1(|f| \chi_B, |\nabla g| \chi_B), \quad x \in B,
\end{equation}
where $B \subset \rn$ is a ball, $f,g \in \mathcal{C}^1(B)$, $\mathcal{I}_1$ is a bilinear potential operator, and $f_B:=\frac{1}{\abs{B}}\int_Bf(x)\,dx$. Inequality \eqref{repformbilinear} arises as a bilinear interpretation of the linear inequality
\begin{equation}\label{repformlin}
|f(x) - f_B| \lesssim I_1(|\nabla f| \chi_B), \quad x \in B, \,f \in \mathcal{C}^1(B),
\end{equation}
where $I_1$ denotes the Riesz potential of order 1. Inequality \eqref{repformlin} is usually referred to as a representation formula (for the oscillation $|f(x) - f_B|$). In the linear setting, representation formulas and Poincar\'e inequalities imply embeddings of Campanato-Morrey spaces (see, for instance, \cite{Lu98} for such embeddings in the Carnot-Carath\'eodory framework). As proved in \cite{MMN}, via \eqref{repformbilinear}, the bilinear analogs to these embeddings come in the form of Campanato-Morrey Leibniz-type rules. More precisely, in the scale of Campanato-Morrey spaces ($\mathcal{L}^{p,\lambda}(w)$ and $L^{q,\lambda}(w)$ below), a typical weighted Leibniz-type rule takes the form (see \cite{MMN})
\begin{equation}\label{leibcampanatomorrey}
\norm{fg}{{L}^{q,\lambda}(w)} \lesssim \norm{\nabla f}{\mathcal{L}^{p_1,\lambda_1}(u)} \norm{g}{\mathcal{L}^{p_2,\lambda_2}(v)} +  \norm{ f}{\mathcal{L}^{p_1,\lambda_1}(u)} \norm{\nabla g}{\mathcal{L}^{p_2,\lambda_2}(v)},
\end{equation}
for (a large class of) weights $u, v, w$ and indices $q, \lambda, p_1, \lambda_1$, $p_2$, and $\lambda_2$. In the unweighted case, the natural scaling for \eqref{leibcampanatomorrey} turns out to be the bilinear Sobolev scaling
\begin{equation}\label{Sobolevscale}
\frac{1}{q}  = \frac{1}{p_1} + \frac{1}{p_2} - \frac{1}{n} \quad  \text{ with }  1 < p_1, p_2 < \infty.
\end{equation}
 From \eqref{repformbilinear}, it now becomes apparent that the prevailing tools for obtaining inequalities \eqref{leibcampanatomorrey} rely on boundedness properties of suitable bilinear potential-type operators. Thus, in the scale of Campanato-Morrey spaces, bilinear potential-type operators play the role that paraproducts and the bilinear Coifman-Meyer multipliers play in the proofs of the Sobolev-based Leibniz-type rules \eqref{katoponce} and their Besov and Triebel-Lizorkin counterparts. Accordingly, the time-frequency Fourier-based tools in the latter are replaced by real-analysis methods in the former.

 The purpose of this article is to further develop time-frequency and real-analysis approaches that allow to prove new Leibniz-type rules in Sobolev and Campanato-Morrey spaces.  In the rest of this introduction we feature some of the main results as we explain the organization of the manuscript.

 In  Section \ref{sec:tools} we recall some definitions and known results on boundedness properties of bilinear fractional integrals in weighted and unweighted Lebesgue spaces that will be useful for our proofs.

In Section \ref{sec:bpapproxident} we explore the behavior of the
bilinear oscillation $|f(x)g(x) - f_B g_B|$ when the mean-value
operator is replaced by an approximation of the identity $\{S_t\}$.
Our exposition includes the case of the infinitesimal generator $L$
of an analytic semigroup $\{S_t\}_{t>0}$ on $L^2(\re^n)$ (i.e.
$S_t=e^{-tL}$) whose kernel $p_t(x,y)$ has  fast-enough off-diagonal
decay. The quantity   $S_t f = e^{-tL}f$ can be thought of  as an
average version of $f$ at the scale $t$ and plays the role of $f_B$
for some $t=t_B,$ when defining function spaces, such as $BMO_L$ and
$H^1_L,$ which better capture  properties of the  solutions to $L
u=0;$ see for instance \cite{DY1}. In the linear case, the new study
of Sobolev-Poincar\'e inequalities associated to the oscillation $|f
- S_{t_B}f|$ has been successfully carried out in \cite{BJM, JM}
(see also \cite{BB}), yielding  Sobolev-Poincar\'e type inequalities
such as
\begin{equation}\label{sobpoinexpanded}
\left( \frac{1}{|B|} \int_B |f - S_{t_B}f|^q \right)^{1/q} \lesssim \sum\limits_{k \in \na_0} \alpha_k\, r(2^k B) \left(\frac{1}{|2^k B|} \int_{2^k B} |\nabla f|^p \right)^{1/p},
\end{equation}
for suitable choices of indices $1 < p < q$ and sequences
$\{\alpha_k\} \subset [0, \infty)$. As described in \cite{BJM, JM},
the presence of the series expansion on the right-hand side of
\eqref{sobpoinexpanded} accounts for the lack of localization of the
approximation of the identity $\{S_t\}$. In this vein, we study
bilinear oscillations of the type $|fg - S_{t_B}f S_{t_B}g|$ and
establish bilinear Poincar\'e-type inequalities in the Euclidean
setting associated to a general approximation of the identity
$\{S_t\}$.   We prove:
\begin{theor} \label{thm:mainpoincare} Let ${\mathcal S}:=\{S_t\}_{t>0}$  and ${\mathcal S'}:=\{t\partial_t S_t\}_{t>0}$ be approximations of
the identity in $\re^n$ of order $m>0$ and constant $\varepsilon$ in \eqref{gammadecay},
 $1<p_1,p_2<\infty,$ $q>0,$ and  $0<\alpha <\min\{1,\varepsilon\}$   such that
$\frac{1}{q}=\frac{1}{p_1}+\frac{1}{p_2} - \frac{1-\alpha}{n}.$
\begin{align*}
\lefteqn{\left(\int_B|f g - S_{r(B)^m}(f) S_{r(B)^m}(g)|^q\right)^{1/q} } \\
& &  \lesssim\, r(B)^\alpha \sum_{l\geq 0} 2^{-l(\varepsilon-\alpha)} \left[ \left(\int_{2^{l+1}B} |\nabla f|^{p_1}\right)^{1/p_1} \left(\int_{2^{l+1}B}|g| \right)^{1/p_2}\right. \\
& & +\left. \left(\int_{2^{l+1}B} | f|^{p_1}\right)^{1/p_1} \left(\int_{2^{l+1}B}|\nabla g| \right)^{1/p_2}\right] .
\end{align*}
\end{theor}
\noindent If fact, our full result is a  more general weighted version of Theorem \ref{thm:mainpoincare} (see Theorem \ref{thm:bp}).
The proof consists in establishing a bilinear representation formula tailored to the semigroup $\{S_t\}$, which, as expected, turns out to be an expanded version of \eqref{repformbilinear} (see Theorem \ref{thm:representation}). This bilinear representation formula involves logarithmic perturbations of the bilinear fractional integral used in \cite{MMN}, whose kernels are proved to still satisfy appropriate growth conditions that guarantee boundedness of the operator on products of weighted Lebesgue spaces.

In Section \ref{sec:campanato} we  define bilinear Campanato-Morrey
spaces associated to $\{S_t\}$ and use the results of
section~\ref{sec:bpapproxident}  to produce associated (weighted)
Leibniz-type rules.

 In Section \ref{sec:extensions} we point out relevant extensions to the contexts of doubling Riemannian manifolds and Carnot groups.

In Section \ref{sec:bpseudo} we close the circle of ideas developed in Sections \ref{sec:intro}-\ref{sec:bpapproxident} by relating bilinear
 pseudo-differential operators  and bilinear potential operators.
More specifically, we study bilinear pseudo-differential operators
of the form
 \begin{align} \label{operator}
  T_\sigma(f,g)(x) & = \int_{\re^{2n}} e^{ix(\xi+\eta)} \sigma(x,\xi,\eta)\widehat{f}(\xi) \widehat{g}(\eta) d\xi d\eta, \quad f,g\in\mathcal{S}(\re^n), \,x\in \re^n.
\end{align}
We relate such operators to potential operators via the inequalities
 \begin{equation}\label{newbond}
 |T_\sigma(f,g)| \lesssim \mathcal{B}_s(|f|,|g|) \quad \text{ and } \quad |T_\sigma(f,g)| \lesssim \mathcal{I}_s(|f|,|g|), \quad f, g \in \Sw,
 \end{equation}
where $\mathcal{B}_s$ is the bilinear fractional integral of order
$s$, introduced and studied in \cite{G} and \cite{KS},
$\mathcal{I}_s$ is the bilinear Riesz potential of order $s$
introduced in \cite{KS}, and $\sigma$  belongs to standard classes
of bilinear symbols of order $-s$. As a consequence of these  bonds
between bilinear pseudo-differential and potential operators, we
obtain the following (see Sections \ref{sec:tools} and
\ref{sec:bpseudo} for pertinent definitions):
\begin{theor}\label{thm:mainpdo}
Suppose $n\in\na$  and consider exponents $p_1,p_2\in(1,\infty)$  and $q,\,s>0$ that satisfy
\begin{equation}\label{Sobolevscales}
\frac{1}{q}  = \frac{1}{p_1} + \frac{1}{p_2} - \frac{s}{n}.
\end{equation}
\begin{enumerate}[(a)]
\item  If  $s\in(0,2n),$ $0\le \delta<1,$ and  $\sigma \in BS^{-s}_{1,\delta}(\re^n)\cup \dot{BS^{-s}_{1,\delta}}(\re^n)$ then
$T_\sigma$ is bounded from $L^{p_1}\times L^{p_2}$ into $L^q$.
\item  If $s\in (0,n),$ $\theta\in (0,\pi)\setminus \{\pi/2,3\pi/4\}$, $0\leq \delta<1$
and $\sigma \in BS^{-s}_{1,\delta;\theta}(\re^n)\cup \dot{ BS^{-s}_{1,\delta;\theta}}(\re^n)$
then the bilinear operator $T_\sigma$ is bounded from $L^{p_1} \times L^{p_2}$ into $L^q.$
\end{enumerate}
\end{theor}
\noindent We actually prove a more general weighted version of
Theorem \ref{thm:mainpdo}, for which we refer the reader to Theorem
\ref{thm:pdobounds} for a more precise statement.  In
section~\ref{sec:leibnizsob} we present some consequences of Theorem
\ref{thm:mainpdo}, such as  Sobolev-based fractional Leibniz rules of the
form
\begin{equation}\label{sobolevkatoponce}
\norm{fg}{W^{m,q}} \lesssim \norm{f}{W^{s+m,p_1}} \norm{g}{L^{p_2}} + \norm{f}{L^{p_1}}\norm{g}{W^{s+ m,p_2}},
\end{equation}
for a range of indices $m$ and $s$, with the novelty $0 < q < 1$,
under the bilinear Sobolev scaling of equation
\eqref{Sobolevscales}.  The relation \eqref{Sobolevscales} sheds
additional light onto the balance between integrability and
smoothness built into inequalities of the type \eqref{katoponce}.
Notice that inequality \eqref{sobolevkatoponce} cannot be obtained
by applying Sobolev embedding and then the fractional Leibniz rule
\eqref{katoponce}, because equation \eqref{Sobolevscales} allows for
$0<q<1$ in which case the Sobolev embedding $W^{m+s,r}\subseteq
W^{m,q}$ fails for any choice of $r$ with $1/r=1/p_1 + 1/p_2$ and
$q$ as in \eqref{Sobolevscales}.

The bilinear Poincar\'e estimates introduced in \cite{MMN} rely on
the oscillation of the pointwise product of two functions (i.e.
$T_\sigma$ with $\sigma \equiv 1$); in turn, they give rise to
bilinear Sobolev inequalities of the form
\begin{equation}
 \|fg\|_{L^q} \lesssim \|\nabla f\|_{L^{p_1}} \|g\|_{L^{p_2}} + \|f\|_{L^{p_1}}\|\nabla g\|_{L^{p_2}} \label{eq:sob},
  \end{equation}
for exponents $p_1,p_2>1$ and $q>0$ satisfying the Sobolev relation
\eqref{Sobolevscale}. These results correspond to the limit of
bilinear Poincar\'e inequalities on balls, by making the radius of
the ball tend to infinity. We direct the reader to  \cite{MMN, Kabe,
MN} for other versions of \eqref{eq:sob}, including weights and
higher order derivatives in the context of H\"ormander vector
fields. The results presented in Section \ref{sec:leibnizsob}
further substantiates inequalities of the type \eqref{eq:sob} under
Sobolev scaling and unifies their study in the language of bilinear
pseudo-differential operators.

 Throughout the paper, we use upper-case letters to label theorems corresponding to known results
while we use single numbers (with no reference to the section) for
theorems, propositions and corollaries that are new and proved in
this article.

\section{Bilinear fractional integrals and their boundedness properties on weighted Lebesgue spaces}\label{sec:tools}

Given a weight $w$ defined on $\re^n$ and $p > 0$, the notation $L^p_w$ will be used to refer to the weighted Lebesgue space of all functions $f:\re^n\to\mathbb{C}$ such that $\|f\|_{L^p_w}^p:=\int_{\re^n}\abs{f(x)}^pw(x)\,dx<\infty$, when $w\equiv 1$ we will simply write $L^p.$

  If $w_1,\,w_2$ are weights defined on $\re^n,$  $1< p_1,\,p_{2}<\infty,$ $q>0,$
  and $  w:=w_1^{q/p_1}w_2^{q/p_2},$
we say that $(w_1,w_2)$ satisfies the $A_{(p_1,p_2),q}$ condition (or that $(w_1,w_2)$ belongs to the  class $A_{(p_1,p_2),q}$) if
$$[(w_1,w_2)]_{A_{(p_1,p_2),q}}:= \sup_{B }\Big(\frac{1}{|B|}\int_B w(x)\,dx\Big)\,
\prod_{j=1}^2 \Big(\frac{1}{|B|}\int_B w_j(x)^{1-p'_j}\,dx
\Big)^{\frac{q}{p'_j}}<\infty,$$
where the supremum is taken over all Euclidean balls $B\subset \re^n$ and $\abs{B}$ denotes the Lebesgue measure of $B$.

The  classes $A_{(p_1,p_2),q}$ are  inspired in the classes of weights $A_{p,q},$ $1\le p,\,q<\infty,$ defined by Muckenhoupt and Wheeden in \cite{MW} to study weighted norm inequalities for the fractional integral: a weight $u$ defined on $\re^n$ is in the class $A_{p,q}$ if
$$\sup_B \left(\frac{1}{|B|}\int_{B} u^{\frac{q}{p}}\,dx\right)\left(\frac{1}{|B|}\int_B u^{(1-p')}\,dx\right)^{\frac{q}{p'}}<\infty.
$$
The classes $A_{(p_1,p_2),q}$ for $1/q=1/p_1+1/p_2$  were introduced in \cite{LOPTT} to study characterizations of weights for boundedness properties  of certain bilinear maximal functions and bilinear Calder\'on-Zygmund operators in weighted Lebesgue spaces. As shown in \cite{Kabe}, the classes $A_{(p_1,p_2),q}$ characterize the weights rendering analogous bounds for bilinear fractional integral operators .

\begin{remark} \label{rem:weights}  If $(w_1,w_2)$ satisfies the $A_{(p_1,p_2),q}$ condition then $w=w_1^{q/p_1}w_2^{q/p_2}$ and $w_i^{1-p_i'},$ $i=1,2,$ are $A_\infty$ weights as shown in \cite[Theorem 3.6 ]{LOPTT} and \cite[Theorem 3.4]{Kabe}.

\end{remark}

For $\alpha>0$, we consider bilinear fractional integral operators on ${\mathbb R}^n$ of order $\alpha>0$ defined by
\begin{align}
\mathcal{B}_{\alpha} (f,g)(x) &:= \int_{{\mathbb R}^{n}}
 \frac{f(x-s_1 y)g(x-s_2 y)}{|y|^{n-\alpha}}\,dy,\quad x\in\re^n.\label{def:B}\\
\mathcal{I}_{\alpha} (f,g)(x) &:= \int_{{\mathbb R}^{2n}}
 \frac{f(y) g(z)}{(|x-y|+|x-z|)^{2n-\alpha}}\,dydz,\quad x\in\re^n,  \label{def:multfracop}
\end{align}
where $s_1 \ne s_2$ are nonzero real numbers. In the following theorem we summarize  results concerning boundedness properties on weighted and unweighted Lebesgue spaces for the operators $\mathcal{B}_\alpha$ and $\mathcal{I}_\alpha$, which will be useful in some of our proofs.

\begin{thm}\label{fracint} In $\re^n:$
\begin{enumerate}[(a)]
\item \label{thm:Ialpha} \cite{KS, Kabe}  Let $\alpha \in (0,2n),$ $1<p_1,p_2<\infty$ and $q>0$ such that
$ \frac{1}{q}=\frac{1}{p_1}+\frac{1}{p_2}-\frac{\alpha}{n}.$ Then $\mathcal{I}_\alpha$ is bounded from $L^{p_1}_{w_1} \times L^{p_2}_{w_2}$ into $L^q_{w}$ for $ w:= w_1^{q/p_1}w_2^{q/p_2}$ and  pairs of weights $(w_1,w_2)$ satisfying the $A_{(p_1,p_2),q}$ condition.

\item \label{thm:Balphaunweighted} \cite{G, KS} Let $\alpha \in (0,n),$ $1<p_1,p_2<\infty$ and $q>0$ such that
$ \frac{1}{q}=\frac{1}{p_1}+\frac{1}{p_2}-\frac{\alpha}{n}.$ Then $\mathcal{B}_\alpha$ is bounded from $L^{p_1} \times L^{p_2}$ into $L^q$.
\item \label{thm:Balphaweighted}{\rm{[Remark \ref{weightsBalpha}]}} Let $\alpha\in (0,n),$ $1<p_1,p_2<\infty$ such that $1/p:=1/p_1+1/p_2<1$ and $q>1$ such that $1/q=1/p-\alpha/n$.  Then $\mathcal{B}_\alpha$ is bounded from $L^{p_1}_{w_1} \times L^{p_2}_{w_2}$ into $L^q_{w}$ for $ w:= w_1^{q/p_1}w_2^{q/p_2}$ and  weights $w_1,\,w_2$ in $A_{p,q}$.
\end{enumerate}
 \end{thm}

 \begin{remark}\label{weightsBalpha} Part \eqref{thm:Balphaweighted} of Theorem~\ref{fracint} follows from the following observations.
 Muckenhoupt and Wheeden \cite{MW} showed that the linear fractional integral operator
$$I_\alpha f(x):=\int_{\re^n} \frac{f(x-y)}{|y|^{n-\alpha}}\,dy$$
satisfies
$$\left(\int_{\re^n} |I_\alpha f(x)|^q u^{\frac{q}{p}}\,dx\right)^{1/q}\leq C\left(\int_{\re^n} |f(x)|^p u\,dx\right)^{1/p}$$
for $1/q=1/p-\alpha/n,$  $u\in A_{p,q}$ and $p,\,q>1$. Using $p$ and $q$ as in the statement of part \eqref{thm:Balphaweighted} of Theorem \ref{fracint},  let $r=p_1/p$ and $s=p_2/p$, so that $r,s>1$ and $1/r+1/s=1$. By H\"older's inequality
$$|\mathcal{B}_\alpha(f,g)|\lesssim  I_\alpha(|f|^r)^{1/r}I_\alpha(|g|^s)^{1/s},$$
and
\begin{align*}
\left(\int_{\re^n} |\mathcal{B}_\alpha(f,g)|^qw\right)^{1/q}&\leq \left(\int_{\re^n}I_\alpha(|f|^r)^{q/r}I_\alpha(|g|^s)^{q/s}w_1^{\frac{q}{p_1}}w_2^\frac{q}{p_2}\right)^{1/q}\\
&\leq \left(\int_{\re^n}I_\alpha(|f|^r)^{q}w_1^{\frac{q}{p}}\,dx\right)^{1/qr}\left(\int_{\re^n}I_\alpha(|g|^s)^{q}w_2^{\frac{q}{p}}\right)^{1/sq}.
\end{align*}
Using the result of Muckenhoupt and Wheeden, the last inequality is bounded by
$$C\left(\int_{\re^n} |f|^{rp} w_1\right)^{1/rp}\left(\int_{\re^n} |g|^{sp} w_2\right)^{1/sp}=C\left(\int_{\re^n} |f|^{p_1} w_1\right)^{1/p_1}\left(\int_{\re^n}|g|^{p_2} w_2\right)^{1/p_2},$$
which is the desired result.

 \end{remark}

 Multilinear potential operators, of which $\mathcal{I}_\alpha$ is a particular case, were studied in \cite{MMN} in the context of spaces of homogeneous type. We now  briefly recall some those results, as they will be used in the proofs in the next sections.

Let $(X,\rho,\mu)$ be a \emph{space of homogenous type}. That is, $X$ is a non-empty set, $\rho$ is a quasi-metric defined on $X$ that satisfies the quasi-triangle inequality
\begin{equation}\label{quasi}
\rho(x,y)\leq \kappa(\rho(x,z)+\rho(z,y)),\qquad x,y,z\in X,
\end{equation}
for some $\kappa\geq 1,$ and $\mu$ is a Borel measure on $X$ (with respect to the
topology defined by $\rho$) such that there exists a constant $L_0\ge
0$ verifying
\begin{equation}\label{doublingMu}
0<\mu(B_\rho(x,2r)\le L_0\,\mu(B_\rho(x,r))<\infty
\end{equation}
for all $x\in X$ and $0<r<\infty,$ and where $B_\rho(x,r)=\{y\in
X:\rho(x,y)<r\}$ is the $\rho$-ball of center $x$ and radius $r.$ Given a ball $B=B_\rho(x,r)$ and $\theta > 0$ we will usually write $r(B)$ to denote the radius $r$ and $\theta B$ to denote $B_\rho(x, \theta r)$. In the Euclidean setting, this is, when $X=\re^n,$ $\rho$ is Euclidean distance and $\mu$ is Lebesgue measure, we use the notation $B(x,r)$ instead of $B_\rho(x,r)$.

The measure $\mu$ is said to satisfy the {\it reverse doubling property} if for every $\eta>1$ there are constants  $c(\eta)>0$ and $\gamma>0$ such that
\begin{equation}\label{reversedoubling}
\frac{\mu(B_\rho(x_1,r_1))}{\mu(B_\rho(x_2,r_2))}\ge c(\eta) \,\left(\frac{r_1}{r_2}\right)^\gamma,
\end{equation}
whenever $B_\rho(x_2,r_2)\subset B_\rho(x_1,r_1),$ $x_1,\,x_2\in X$ and $0<r_1,\,r_2\le\eta\, \rm{diam}_\rho(X).$

We consider bilinear potential operators of the form
\begin{equation} \label{generalpot}\mathcal{T}(f,g)(x) = \int_{X^2} f(y)g(z) K(x,y,z) \,d\mu(y)d\mu(z),
\end{equation}
where the kernel $K$ is the restriction of a nonnegative continuous kernel $\tilde{K}(x_1,x_2,y,z)$ (i.e. $K(x,y,z)=\tilde{K}(x,x,y,z)$ for $(x,y,z)\in X\times X \times X$) that satisfies the following \emph{growth conditions}: for every $c>1$ there exists $C>1$ such that
\begin{eqnarray}\label{growth}
\tilde{K}(x_1,x_2,y,z)&\leq& C \tilde{K}(v,w,y,z) \quad {\rm if} \ \rho(v,y)+\rho(w,z) \leq c \,(\rho(x_1,y)+\rho(x_2,z)), \text{ and} \nonumber \\ \\
\tilde{K}(x_1,x_2,y,z)&\leq& C \tilde{K}(y,z,v,w) \quad {\rm if} \ \rho(y,v)+\rho(z,w) \leq c \,(\rho(x_1,y)+\rho(x_2,z)). \nonumber
\end{eqnarray}
The functional $\varphi$ associated to $K$ is defined by
$$
\varphi(B):=\sup\{ K(x,y,z) : (x,y,z) \in B\times B\times B, \rho(x,y)+\rho(x,z)\geq c\,r(B)\}
$$
for a sufficiently small positive constant $c$ and for $B$ a $\rho$ ball such that $r(B)\le \eta \,\rm{diam}_\rho(X),$ for some fixed $\eta>1.$
The functional $\varphi$ associated to $K$  will be assumed to satisfy the following property:  there exists $\delta>0$  such that for all $C_1>1$  there exists $C_2>0$ such that
\begin{equation} \label{phiassumption}
\varphi(B')\mu(B')^2 \leq C_2\,\left(\frac{r(B')}{r(B)}\right)^\delta \varphi(B)\mu(B)^2
\end{equation}
for  all balls  $B'\subset B,$ with  $r(B'),\,r(B)< C_1\,  \rm{diam}_\rho(X)$.

We note that, in the Euclidean setting, the operator $\mathcal{I}_\alpha$ defined in \eqref{def:multfracop} has kernel and associated functional given, respectively, by
\[K(x,y,z)=\frac{1}{(\abs{x-y}+\abs{x-z})^{2n-\alpha}}\qquad \text{and}\qquad \varphi(B)\sim r(B)^{\alpha-2n},\]
and both satisfy \eqref{growth} and \eqref{phiassumption}.

\begin{thm}[\cite{MMN}]\label{general2w} Suppose that $1<p_1, p_2\leq \infty,$ $\frac{1}{p}=\frac{1}{p_1}+\frac{1}{p_2}$ and $\frac{1}{2}<p\le q<\infty.$
  Let $(X,\rho,\mu)$ be a space of homogeneous type  that satisfies the reverse doubling property \eqref{reversedoubling} and let  $K$ be a kernel such that \eqref{growth} holds with $\varphi$ satisfying \eqref{phiassumption}.  Furthermore, let $u, v_k,$
$k=1,2$ be weights defined on $X$ that satisfy condition
\eqref{general2wq>1hm} if $q>1$ or condition \eqref{general2wq<1hm} if $q\le 1,$
where
\begin{equation}\label{general2wq>1hm}
\sup_{B\, \rho\text{-ball}} \varphi(B) \mu(B)^{\frac{1}{q} + \frac{1}{{p_1}'}+\frac{1}{{p_2'}}}
\left( \frac{1}{\mu(B)}\int_B u^{qt} d\mu \right)^{1/qt}
\prod_{j=1}^2 \left( \frac{1}{\mu(B)}\int_B v_i^{-tp_i'} d\mu
\right)^{1/tp_i'} < \infty,
\end{equation}
for some $t > 1$,

\begin{equation}\label{general2wq<1hm}
\sup_{B\, \rho\text{-ball}} \varphi(B) \mu(B)^{\frac{1}{q} + \frac{1}{{p_1}'}+\frac{1}{{p_2'}}}
\left( \frac{1}{\mu(B)}\int_B u^{q} d\mu \right)^{1/q} \prod_{j=1}^2
\left( \frac{1}{\mu(B)}\int_B v_i^{-tp_i'} d\mu \right)^{1/tp_i'} <
\infty,
\end{equation}
for some $t > 1, $ with the supremum taken over $\rho$-balls with $r(B)\lesssim \text{diam}_{\rho}(X).$
Then there exists a constant $C$  such that
\[
\left(\int_{X}\left(\abs{\mathcal{T}(f_1,f_2)}u\right)^q\,d\mu\right)^{1/q}\le
C \prod_{k=1}^2\left(\int_X(\abs{f_k}v_k)^{p_k}\,d\mu\right)^{1/p_k}
\]
for all $(f_1,f_2)\in L^{p_1}_{v_1^{p_1}}(X )\times L^{p_2}_{v_2^{p_2}}(X).$ The constant $C$ depends only on the constants appearing in \eqref{quasi}, \eqref{doublingMu}, \eqref{reversedoubling}, \eqref{growth}, \eqref{phiassumption}, \eqref{general2wq>1hm} and \eqref{general2wq<1hm}.
\end{thm}

\begin{remark} \label{remm}   A careful examination of the proof of Theorem \ref{general2w} yields
$$
\|\mathcal{T}\|_{\text{op}}\lesssim \frac{C}{1-D^{-\delta}}\,\mathcal{W},
$$
where  $\mathcal{W}$ is the constant from \eqref{general2wq>1hm} or \eqref{general2wq<1hm},
$C=\max(C_1,C_2)$ and $\delta>0$ are the constants from \eqref{phiassumption}, and $D>1$ is a structural constant.
\end{remark}

\begin{remark} \label{weightscomparison} In the Euclidean setting, consider weights $w_1,\,w_2\in A_{(p_1,p_2),q}$ and $w=w_1^{{q}/{p_1}}w_2^{{q}/{p_2}},$ for some $1<p_1,\,p_2<\infty,$ $0<\frac{1}{q}<\frac{1}{p_1}+\frac{1}{p_2}$  and suppose that $\mathcal{T}$ is an operator of the form \eqref{generalpot} such that
\[
\sup_{B} \varphi(B) |B|^{\frac{1}{q} + \frac{1}{{p_1}'}+\frac{1}{{p_2'}}}\sim\sup_{B} \varphi(B) r(B)^{\frac{n}{q} + \frac{n}{{p_1}'}+\frac{n}{{p_2'}}}<\infty.
\]
It then follows that $u:=w^{\frac{1}{q}}$ and $v_k:=w_k^{\frac{1}{p_k}},$ $k=1,2,$ satisfy \eqref{general2wq>1hm} and \eqref{general2wq<1hm}. Indeed,
the second factor in \eqref{general2wq>1hm} is given by
\begin{align}\label{general2wq>1hm2}
\sup_{\,x\in\re^n,r>0}
\left( \frac{1}{|B(x,r)|}\int_{B(x,r)} w^{t} dx \right)^{1/qt}
\prod_{j=1}^2 \left( \frac{1}{|B(x,r)|}\int_{B(x,r)} w_i^{-\frac{t}{p_i-1}} dx.
\right)^{1/tp_i'}
\end{align}
Since  $w,\,w_1^{-\frac{1}{p_1-1}},\,w_2^{-\frac{1}{p_2-1}}$ are $A_\infty$ weights (see Remark \ref{rem:weights}), there exists $t>1$ such that \eqref{general2wq>1hm2} is bounded by
\[
\sup_{B}
\left( \frac{1}{|B|}\int_{B} w \,dx \right)^{1/q}
\prod_{j=1}^2 \left( \frac{1}{|B|}\int_{B} w_i^{-\frac{1}{p_i-1}} dx
\right)^{1/p_i'}=[(w_1,w_2)]_{A_{(p_1,p_2),q}} <\infty,
\]
where finiteness is due to $(\,w_1,w_2)$ satisfying the $A_{(p_1,p_2),q}$ condition. A similar reasoning applies to \eqref{general2wq<1hm}.
\end{remark}

The last two remarks imply the following

\begin{coro}
In the $n$-dimensional Euclidean setting, consider weights $w_1,\,w_2\in A_{(p_1,p_2),q}$ and $w=w_1^{{q}/{p_1}}w_2^{{q}/{p_2}},$ for some $1<p_1,\,p_2<\infty,$ $0<\frac{1}{q}<\frac{1}{p_1}+\frac{1}{p_2}.$  Suppose that $\mathcal{T}$ is an operator of the form \eqref{generalpot} such that
its kernel satisfies \eqref{growth}, the associated functional  $\varphi$ satisfies \eqref{phiassumption}, and
\[
\sup_{B} \varphi(B) r(B)^{\frac{n}{q} + \frac{n}{{p_1}'}+\frac{n}{{p_2'}}}<\infty.
\]
Then there exists a constant $A$  such that
\[
\left(\int_{\re^n}\abs{\mathcal{T}(f_1,f_2)}^qw\,dx\right)^{1/q}\le
A \prod_{k=1}^2\left(\int_{\re^n}\abs{f_k}^{p_k}w_k\,dx\right)^{1/p_k}
\]
for all $(f_1,f_2)\in L^{p_1}_{w_1}\times L^{p_2}_{w_1}.$ The constant $A$ satisfies
\[A\le c\, \sup_{B} \varphi(B) r(B)^{\frac{n}{q} + \frac{n}{{p_1}'}+\frac{n}{{p_2'}}},\]
where $c$ depends only on  $[(w_1,w_2)]_{A_{(p_1,p_2),q}} $ and other absolute constants.
\end{coro}

\section{Bilinear Poincar\'e-type inequalities relative to an approximation of  the identity}\label{sec:bpapproxident}

An \emph{approximation of the identity}  of order $m>0$ in $\re^n$ is a collection of operators ${\mathcal S}:=\{S_t\}_{t>0}$ acting on functions defined on $\re^n,$
\[
S_tf(x)=\int_{\re^n}p_t(x,y)f(y)\,dy, \quad x \in \rn,
\]
such that for each $t>0$ the kernels $p_t$ satisfy $\int_{\re^n}p_t(x,y)\,dy=1$ for all $x$  and  the scaled Poisson bound
\begin{equation}\label{pbound}
\abs{p_t(x,y)}\le t^{-n/m}\,\gamma\left(\frac{|x-y|}{t^{1/m}}\right), \qquad x,\,y\in\re^n,
\end{equation}
where $\gamma:[0,\infty)\to[0,\infty)$ is a bounded, decreasing function for which
 \begin{equation}\label{gammadecay}
\lim_{r\to\infty} r^{2n+\varepsilon}\gamma(r)=0,\qquad \text{ for some }\, \varepsilon>0.
\end{equation}

 As examples, it is well-known that if  a sectorial operator $L$ generates a holomorphic semigroup $\{e^{-zL}\}_z$ whose  kernels satisfy suitable pointwise bounds, then $S_t=e^{-tL}$ gives rise to an approximation of the identity. The resolvents $S_t=(1+tL)^{-M}$ or $S_t=1-(1-e^{-tL})^N$ can be considered as well.
 We refer the reader to \cite{DY1} and \cite{Mc} for more details concerning holomorphic functional calculus.
Other  examples can be built on a second-order divergence form operator $L=-\textrm{div} (A \nabla)$ with an elliptic matrix-valued function $A$. Since $L$ is maximal accretive, it admits a bounded $H_\infty$-calculus on $L^2(\re^n)$. Moreover, when $A$ has real entries or when the dimension $n\in\{1,2\}$, then the operator $L$ generates an analytic semigroup on $L^2$ with a heat kernel satisfying Gaussian upper-bounds.

The main result of this section is the following:

\begin{theor} \label{thm:bp} Let ${\mathcal S}:=\{S_t\}_{t>0}$  and ${\mathcal S'}:=\{t\partial_t S_t\}_{t>0}$ be approximations of the identity in $\re^n$ of order $m>0$ and constant $\varepsilon$ in \eqref{gammadecay},
 $1<p_1,p_2<\infty,$ $q>0,$ and  $0<\alpha <\min\{1,\varepsilon\}$   such that
$\frac{1}{q}=\frac{1}{p_1}+\frac{1}{p_2} - \frac{1-\alpha}{n}.$
If   $(w_1,w_2)$ satisfy the $A_{(p_1,p_2),q}$ condition and $w:=w_{1}^{q/p_1} w_2^{q/p_2}$, then  there exists a  constant $C$ such that for all Euclidean  balls $B$
\begin{align*}
\lefteqn{\left\|f g - S_{r(B)^m}(f) S_{r(B)^m}(g)\right\|_{L^q_w(B)}} & & \\
& &  \le C\, r(B)^\alpha \sum_{l\geq 0} 2^{-l(\varepsilon-\alpha)} \left[\left\|\nabla f \right\|_{L_{w_1}^{p_1}(2^{l+1}B)} \left\|g \right\|_{L^{p_2}_{w_2}(2^{l+1}B)} + \left\|f\right\|_{L^{p_1}_{w_1}(2^{l+1} B)} \left\|\nabla g\right\|_{L^{p_2}_{w_2}(2^{l+1}B)} \right].
\end{align*}
\end{theor}

\begin{remark} It is possible to  consider two collections of operators ${\mathcal S}^1:=\{S^1_t\}_{t>0}$ and ${\mathcal S}^2:=\{S^2_t\}_{t>0}$, then the proof of Theorem \ref{thm:bp} holds true when estimating the oscillation
$ \|f g - S^1_{r(B)^m}(f) S^2_{r(B)^m}(g)\|_{L^q_w(B)}.$
\end{remark}

\begin{remark} Note that condition \eqref{gammadecay} assumes exponent $2n+\varepsilon$ rather than $n+\varepsilon.$  This is quite natural in our context since the proof of Theorem \ref{thm:bp} involves   the semigroup ${\mathcal P_t}:=S_t\otimes S_t$  which is expected to have decay for $2n$-dimensional variables.
\end{remark}

 \begin{remark}  The scaling of the result in Theorem~\ref{thm:bp}  is in accordance with the  classical situation corresponding to $\alpha=0$ and obtained in \cite{MMN}.
More precisely, a particular case of \cite[Theorem 1]{MMN} reads
\begin{equation}\label{particularcase}
\left\|f g - f_B g_B\right\|_{L^q_w(B)}  \le C\, (\left\|\nabla f \right\|_{L_{w_1}^{p_1}(B)} \left\|g \right\|_{L^{p_2}_{w_2}(B)} + \left\|f\right\|_{L^{p_1}_{w_1}(B)} \left\|\nabla g\right\|_{L^{p_2}_{w_2}(B)})
\end{equation}
for $\frac{1}{q}=\frac{1}{p_1}+\frac{1}{p_2}-\frac{1}{n}$, with $\frac{1}{p_1}+\frac{1}{p_2}  < 2$, $(w_1,w_2)\in A_{(p_1,p_2),q}$ and $w=w_1^{q/{p_1}}w_2^{q/{p_2}}$.  H\"older's inequality and the conditions on the weights imply
\[
\left\|f g - f_B g_B\right\|_{L^q_w(B)}  \le C\, r(B)^\alpha\, (\left\|\nabla f \right\|_{L_{w_1}^{p_1}(B)} \left\|g \right\|_{L^{p_2}_{w_2}(B)} + \left\|f\right\|_{L^{p_1}_{w_1}(B)} \left\|\nabla g\right\|_{L^{p_2}_{w_2}(B)})
\]
for  $\frac{1}{q}=\frac{1}{p_1}+\frac{1}{p_2}-\frac{1-\alpha}{n},$ $(w_1,w_2)\in A_{(p_1,p_2),q}$ and $w=w_1^{q/{p_1}}w_2^{q/{p_2}}$.

 For instance, let $p_1,\,p_2,\,q,\,\alpha,$ $w_1,\,w_2,\,w$ be as in the statement of Theorem~\ref{thm:bp}. Define $\bar{q}$ by $\frac{1}{\bar{q}}=\frac{1}{p_1}+\frac{1}{p_2}-\frac{1}{n}=\frac{1}{q}-\frac{\alpha}{n}$ and assume that $\bar{q}>0$ and that the pair $(w_1, w_2)$ is in $A_{(p_1,p_2),\bar{q}}.$ Setting $\bar{w}=w_1^{\bar{q}/{p_1}}w_2^{\bar{q}/p_2}$ and using H\"older's inequality and \eqref{particularcase}  we obtain
 \begin{align*}
 \left\|f g - f_B g_B\right\|_{L^q_w(B)}  &\le \left(\int_B\bar{w}\,\left(w_1^{\frac{q-\bar{q}}{p_1}}w_2^{\frac{q-\bar{q}}{p_2}}\right)^{(\frac{\bar{q}}{q})'}\right)^{\frac{1}{q(\bar{q}/q)'}} \, \left\|f g - f_B g_B\right\|_{L^{\bar{q}}_{\bar{w}(B)}}\\
 &\lesssim r(B)^\alpha\,(\left\|\nabla f \right\|_{L_{w_1}^{p_1}(B)} \left\|g \right\|_{L^{p_2}_{w_2}(B)} + \left\|f\right\|_{L^{p_1}_{w_1}(B)} \left\|\nabla g\right\|_{L^{p_2}_{w_2}(B)}).
 \end{align*}

 Note, however, that Theorem \ref{thm:bp} does not include the case $\alpha=0.$

 \end{remark}

\begin{remark}
Since we do not require spatial regularity on the kernels $p_t$ in \eqref{pbound}, our results can be extended  to every subset of ${\mathbb R^n}$ (not necessarily Lipschitz) by considering truncations as used in \cite{DY2}.
\end{remark}

Our proof of Theorem \ref{thm:bp} is based on an appropriate representation formula for the bilinear oscillations associated to the approximation of the identity and the boundedness properties of operators studied in \cite{MMN}. We present the details in the next two subsections.

\subsection{Representation formula.}\label{sec:repformula}
We start by introducing the collection of bilinear operators that shape our representation formula. For a ball $B\subset \re^n$, the operator ${\mathcal J}_{B}$ is defined as
\begin{equation}
{\mathcal J}_{B}(f_1,f_2)(x) :=\int_{B\times B} K(x,(a,b))  f_1(a)f_2(b)\, da \, db \quad x\in B , \label{eq:J} \end{equation}
with kernel
\[
K({x}, (a,b)):=\frac{1}{(\abs{x-a}+\abs{x-b})^{2n-1}}\log\left(\frac{8\,r(B)}{\abs{x-a}+\abs{x-b}}\right), \qquad  x,\,a,\,b\in B.
\]

\begin{theor}[Bilinear representation formula] \label{thm:representation}
Let $\mathcal{S}=\{S_t\}_{t>0}$  and $\mathcal{S}'=\{t\partial_t S_t\}_{t>0}$ be approximations of the identity in $\re^n$ of order $m>0$ and constant $\varepsilon$ in \eqref{gammadecay}. There exists a constant $C>0$ such that for every ball $B\subset \re^n$ and $x\in B,$
\begin{align*}
\lefteqn{\left|f(x) g(x) - S_{r(B)^m}(f)(x) S_{r(B)^m}(g)(x)\right|}  & \\
&   \le\, C\, \sum_{l\geq 0} 2^{-l\varepsilon} \left[\mathcal{J}_{2^{l+1}B}(|\nabla f|\chi_{2^{l+1}B},|g|\chi_{2^{l+1}B})(x)+ {\mathcal J}_{2^{l+1}B}(|f|\chi_{2^{l+1} B},|\nabla g|\chi_{2^{l+1}B})(x)\right].
\end{align*}
\end{theor}

\begin{remark} As mentioned in the Introduction, since the approximation operator $S_{r(B)^m}$ is not a local operator, we cannot expect perfectly localized estimates as for the ``classical'' Poincar\'e inequality.
\end{remark}

\begin{proof}
We consider  the operator on $\re^{2n}$ given by $\mathcal{P}_t  := S_t \otimes S_t$, that is,
$$
\mathcal{P}_t(F)(x,x) = \int_{\re^n} \int_{\re^n} p_t(x,y) p_t(x,z) F(y,z) \, dy dz.
$$
For given functions $f$ and $g$ defined on $\re^n$, let $F(y,z):=f(y) g(z).$
Fix $B$ of radius $r(B),$ $x\in B$ and for each $t\in (0,r(B)^m)$ let $B_t$ be the ball of radius $t^{1/m}$ centered at $x\in \re^n$.
Then
\begin{align*}
F(x,x) - \mathcal{P}_{r(B)^m}(F)(x,x) & = - \int_0^{r(B)^m} t \partial_t \mathcal{P}_t(F)(x,x) \, \frac{dt}{t}\\
 & =  - \int_0^{r(B)^m} t \partial_t  \mathcal{P}_t(F- F_{B_t \times B_t})(x,x)\, \frac{dt}{t},
\end{align*}
where we used that $F_{B_t \times B_t}= f_{B _t}g_{B_t}$ is a constant and $\partial_t S_t(1) = 0$ for all $t > 0$.
The pointwise bounds \eqref{pbound} for the kernels $p_t(x,y)$ give
\begin{align*}
\left|t \partial_t \mathcal{P}_t  (F - F_{B_t \times B_t})(x,x) \right|& \\
& \hspace{-4cm} \lesssim \int_{\re^n} \int_{\re^n} {t^{-\frac{2n}{m}}}  \left(1+\frac{|x-y|}{t^{\frac{1}{m}}} \right)^{-2n-\varepsilon} \left(1+\frac{|x-z|}{t^{\frac{1}{m}}} \right)^{-2n-\varepsilon} |f(y)g(z) - f_{B_t} g_{B_t}| \, dy dz \\
&  \hspace{-4cm} \lesssim  \iint\limits_{B_t \times B_t} t^{-\frac{2n}{m}} \, \left(1+\frac{|x-y|}{t^{\frac{1}{m}}} \right)^{-2n-\varepsilon} \left(1+\frac{|x-z|}{t^{\frac{1}{m}}} \right)^{-2n-\varepsilon} |f(y)g(z) - f_{B_t} g_{B_t}| \, dy dz \\
& \hspace{-4cm} + \sum_{l \in \na}  \ \iint\limits_{C_l(B_t \times B_t)} t^{-\frac{2n}{m}} \,   \left(1+\frac{|x-y|}{t^{\frac{1}{m}}} \right)^{-2n-\varepsilon} \left(1+\frac{|x-z|}{t^{\frac{1}{m}}} \right)^{-2n-\varepsilon}|f(y)g(z) - f_{B_t} g_{B_t}| \, dy dz\\
& \hspace{-4cm} =: I_0(f,g,t)(x) + \sum_{l\in \na} I_l(f,g,t)(x),
\end{align*}
where for $l\geq 1$, $C_l(B_t\times B_t)$ denotes the annulus
$$ C_l(B_t \times B_t):= 2^{l}(B_t \times B_t) \setminus 2^{l-1}(B_t \times B_t).$$
We now proceed to estimating each of the terms $I_l(f,g,t),$ $l\ge 0.$

\subsection*{The bound for $I_0(f,g,t)$.}

Notice that for all $y, z \in B_t,$
\begin{equation}\label{repfg}
|f(y)g(z) - f_{B_t} g_{B_t}| \lesssim \iint_{B_t \times B_t}  \frac{|\nabla f(a)||g(b)| + |f(a)||\nabla g(b)| }{(|y-a|+|z-b|)^{2n-1}} \, da db.
\end{equation}
Indeed, the usual representation formula for a linear oscillation in $(\re^n)^2$ gives
$$
|F(y,z) - F_{B_t\times B_t}| \leq C \int_{B_t\times B_t} \frac{|\nabla F(a,b)|}{|(y,z)-(a,b)|^{2n-1}} dadb
$$
which yields \eqref{repfg}. Hence, we get
\begin{align*}
I_0(f,g,t)(x) & \lesssim \iint\limits_{B_t \times B_t} t^{-2n/m} |f(y)g(z) - f_{B_t} g_{B_t}| \, dy dz\\
& \lesssim \iint\limits_{B_t \times B_t} (|\nabla f(a)||g(b)| + |f(a)||\nabla g(b)|) I(a,b,t) \, da \,db,
\end{align*}
where
$$
I(a,b,t) := \ \ \iint\limits_{B_t \times B_t} \frac{t^{-2n/m}}{(|y-a|+|z-b|)^{2n-1}} \,dy \,dz.
$$
For $a,b\in B_t,$ we have
\begin{align}
I(a,b,t)  & \leq \
\iint\limits_{\genfrac{}{}{0pt}{}{|y-a|\leq 2t^{1/m}}{|z-b|\leq 2 t^{1/m}}} \ \frac{1}{(|y-a|+|z-b|)^{2n-1}}\frac{dy}{t^{n/m}} \frac{dz}{t^{n/m}} \nonumber \\
& \lesssim \int_{0}^{2t^{1/m}} \int_{0}^{2t^{1/m}} \frac{1}{(u+v)^{2n-1}}u^{n-1} v^{n-1} \frac{du}{t^{n/m}} \frac{dv}{t^{n/m}} \nonumber \\
\nonumber \\
& \lesssim t^{(-2n+1)/m} \int_{0}^{1} \int_{0}^{1} \frac{u^{n-1} v^{n-1}}{(u+v)^{2n-1}}  du dv \lesssim t^{(-2n+1)/m}, \label{eq:int}
\end{align}
where the last integral is controlled by separately estimating for $v\geq u$ and for $u \geq v$.
We conclude that, for $a,b\in B_t$,
\begin{align}
\iint\limits_{B_t\times B_t} \frac{1}{(|y-a|+|z-b|)^{2n-1}}\frac{dy}{t^{n/m}} \frac{dz}{t^{n/m}}
& \lesssim t^{(-2n+1)/m} \lesssim \frac{1}{(|x-a|+|x-b|)^{2n-1}}, \label{eq:aze} \end{align}
and therefore
 \begin{align*}
 I_0(f,g,t)(x) & \lesssim \iint\limits_{B_t \times B_t}  \frac{|\nabla f(a)||g(b)| + |f(a)||\nabla g(b)|}{(|x-a|+|x-b|)^{2n-1}} \, da \,db.
\end{align*}
Integration with respect to the variable $t\in(0,r(B)^m)$ yields,
\begin{align*}
& \int_0^{r(B)^m} I_0(f,g,t)(x) \, \frac{dt}{t}  \lesssim \int_0^{r(B)^m} \iint\limits_{B_t \times B_t} \frac{|\nabla f(a)||g(b)| + |f(a)||\nabla g(b)|}{(|x-a|+|x-b|)^{2n-1}} \, da \,db \, \frac{dt}{t} \\
& \lesssim \int_0^{r(B)^m} \iint\limits_{\genfrac{}{}{0pt}{}{|x-a|\leq t^{1/m}}{|x-b|\leq t^{1/m}}} \frac{|\nabla f(a)||g(b)| + |f(a)||\nabla g(b)|}{(|x-a|+|x-b|)^{2n-1}} \, da \,db \, \frac{dt}{t} \\
& \lesssim \iint\limits_{{2B\times 2B}}
\int_{\genfrac{}{}{0pt}{}{0\leq t\leq r(B)^m}{\genfrac{}{}{0pt}{}{|x-a|\leq t^{1/m}}{|x-b|\leq t^{1/m}}}}
 \frac{|\nabla f(a)||g(b)| + |f(a)||\nabla g(b)|}{(|x-a|+|x-b|)^{2n-1}} \, \frac{dt}{t} \, da \,db\\
& \lesssim \iint\limits_{2B\times 2B} \frac{|\nabla f(a)||g(b)| + |f(a)||\nabla g(b)|}{(|x-a|+|x-b|)^{2n-1}} \log\left( 1+\frac{r(B)^m}{\max\{|x-a|^m,|x-b|^m\}} \right) \, da \,db \\
& \lesssim \iint\limits_{2B\times 2B} \frac{|\nabla f(a)||g(b)| + |f(a)||\nabla g(b)|}{(|x-a|+|x-b|)^{2n-1}} \log\left( \frac{16 r(B)}{|x-a|+|x-b|} \right) \, da \,db \\
& \lesssim {\mathcal J}_{2B}(|\nabla f|,|g|)(x)+ {\mathcal J}_{2B}(|f|,|\nabla g|)(x),
\end{align*}
where the operator ${\mathcal J}_{2B}$ was defined in \eqref{eq:J}. It remains to  treat the terms $I_l(f,g,t)(x)$ with $l\geq 1$.

\subsection*{The bound for $I_l(f,g,t)$ with $l\geq 1$.}

Recall that  $I_l$ is given by
\begin{align*}
 I_l(f,g,t)(x) & := \iint\limits_{C_l(B_t \times B_t)} t^{-2n/m} \left(1+\frac{|x-y|}{t^{\frac{1}{m}}} \right)^{-2n-\varepsilon} \left(1+\frac{|x-z|}{t^{\frac{1}{m}}} \right)^{-2n-\varepsilon}\\
 & \times |f(y)g(z) - f_{B_t} g_{B_t}| \, dy dz,
\end{align*}
where $B_t=B(x,t^{1/m})$ (and therefore $x\in B_t$) and $C_l(B_t \times B_t):=2^{l}(B_t \times B_t) \setminus 2^{l-1}(B_t \times B_t)$. We have to estimate the oscillation $|f(y)g(z) - f_{B_t} g_{B_t}|,$ with $(y,z)\in C_l(B_t\times B_t)$, for which we consider the intermediate averages as follows:
 $$|f(y)g(z) - f_{B_t} g_{B_t}| \leq |f(y)g(z) - f_{2^{l} B_t} g_{2^{l} B_t}| + \sum_{k=0}^{l-1} \left|f_{2^{k+1} B_t} g_{2^{k+1} B_t}-f_{2^{k} B_t} g_{2^{k} B_t} \right|.$$
For all $k\in 0,..., l-1$, we use
 \begin{align*}
  & \left|f_{2^{k+1} B_t} g_{2^{k+1} B_t}-f_{2^{k} B_t} g_{2^{k} B_t} \right|  \lesssim (2^k t^{1/m})^{-2n}\  \iint\limits_{2^{k}B_t \times 2^k B_t} \ |f(u)g(v) - f_{2^{k+1} B_t} g_{2^{k+1} B_t}| du dv \\
 & \lesssim (2^k t^{1/m})^{-2n} \ \iint\limits_{2^{k+1}B_t \times 2^{k+1} B_t} \ |f(u)g(v) - f_{2^{k+1} B_t} g_{2^{k+1} B_t}| du dv.
 \end{align*}
 As done in \eqref{repfg} applied to the ball $2^{k+1} B_t$, we obtain that for $(u,v)\in 2^{k+1}B_t \times 2^{k+1} B_t$
 $$\left|f(u)g(v) - f_{2^{k+1} B_t} g_{2^{k+1} B_t} \right| \lesssim \ \iint\limits_{2^{k+1} B_t \times 2^{k+1} B_t} \ \frac{|\nabla f(a)||g(b)| + |f(a)||\nabla g(b)| }{(|u-a|+|v-b|)^{2n-1}} \, da db.$$
 Proceeding as in \eqref{eq:aze}, by replacing the ball $B_t$ with $2^{k+1} B_t$, we have that for $(a,b)\in 2^{k+1} B_t \times 2^{k+1} B_t$ and $(u,v)\in 2^{k+1} B_t \times 2^{k+1} B_t$,
 \begin{equation} \iint\limits_{2^{k+1}B_t \times2^{k+1} B_t} \ \frac{2^{-2kn} t^{-2n/m} dy dz}{(|u-a|+|v-b|)^{2n-1}} \lesssim (2^k t^{1/m})^{-(2n-1)} \lesssim (|x-a|+|x-b|)^{1-2n}. \label{equa} \end{equation}
Combining everything we have
 \begin{align*}
 & \left|f_{2^{k+1} B_t} g_{2^{k+1} B_t}-f_{2^{k} B_t} g_{2^{k} B_t} \right|\\
  & \lesssim \frac{2^{2kn} t^{2n/m}}{(2^k t^{1/m})^{2n}}
 \ \iint\limits_{2^{k+1} B_t \times 2^{k+1} B_t} \ \frac{|\nabla f(a)||g(b)| + |f(a)||\nabla g(b)| }{(|x-a|+|x-b|)^{2n-1}} \, da db \\
 & \lesssim \ \iint\limits_{2^{k+1} B_t \times 2^{k+1} B_t} \ \frac{|\nabla f(a)||g(b)| + |f(a)||\nabla g(b)| }{(|x-a|+|x-b|)^{2n-1}} \, da db.
 \end{align*}
We conclude that for $(y,z)\in 2^{l+1}(B_t \times B_t)$ (actually for any $y$ and $z$)
 \begin{align}
  |f(y)g(z) - f_{B_t} g_{B_t}| & \lesssim  |f(y)g(z) - f_{2^{l} B_t} g_{2^{l} B_t}|  \nonumber \\
 & \ +\sum_{k=0}^{l-1} \ \ \iint\limits_{2^{k+1} B_t \times 2^{k+1} B_t}  \frac{|\nabla f(a)||g(b)| + |f(a)||\nabla g(b)| }{(|x-a|+|x-b|)^{2n-1}} \, da db. \label{eq:poinc}
 \end{align}
Consequently,
 \begin{align*}
  I_l(f,g,t)(x) & \lesssim I_l^1(f,g,t)(x)+I_l^2(f,g,t)(x)
\end{align*}
with
\begin{align*}
 I_l^1(f,g,t)(x):= & \ \iint\limits_{C_l(B_t \times B_t)} \ \ \left[\left(1+\frac{|x-y|}{t^{\frac{1}{m}}} \right) \left(1+\frac{|x-z|}{t^{\frac{1}{m}}} \right)\right]^{-2n-\varepsilon} \\
 &\times  |f(y)g(z) - f_{2^{l} B_t} g_{2^{l} B_t}|  \frac{dydz}{t^{2n/m}}
\end{align*}
and
\begin{align*}
I_l^2(f,g,t)(x):= & \sum_{k=0}^{l} \ \ \iint\limits_{C_l(B_t \times B_t)} \ \ \Bigg[ \ \ \iint\limits_{2^{k+1} B_t \times 2^{k+1} B_t} \ \left(1+\frac{|x-y|}{t^{\frac{1}{m}}} \right)^{-2n-\varepsilon} \left(1+\frac{|x-z|}{t^{\frac{1}{m}}} \right)^{-2n-\varepsilon} \\
 & \hspace{0.5cm}   \frac{|\nabla f(a)||g(b)| + |f(a)||\nabla g(b)| }{(|x-a|+|x-b|)^{2n-1}} \, da db \Bigg] \frac{dydz}{t^{2n/m}}.
\end{align*}
The first term $I_l^1(f,g,t)(x)$ can be estimated in the same way  as the quantity $I_0(f,g,t)$ by replacing $B_t$ with $2^l B_t$. Since $(y,z)\in C_l(B_t \times B_t)$ and $x\in B_t$, the term
$$\left(1+\frac{|x-y|}{t^{\frac{1}{m}}} \right)^{-2n-\varepsilon} \left(1+\frac{|x-z|}{t^{\frac{1}{m}}} \right)^{-2n-\varepsilon}$$
provides an extra factor $2^{-l(\varepsilon+2n)}$ which partially compensates the normalization coefficient  $2^{2ln}$. So we have
\begin{align*}
\int_0^{r(B)^m} I_l^1(f,g,t)(x)  \frac{dt}{t} & \lesssim 2^{-l\varepsilon} \Big[\mathcal{J}_{2^{l+1}B}(|\nabla f|\chi_{2^{l+1}B},|g|\chi_{2^{l+1}B})(x)  \\
& \hspace{1cm} + {\mathcal J}_{2^{l+1}B}(|f|\chi_{2^{l+1} B},|\nabla g|\chi_{2^{l+1}B})(x)\Big].
\end{align*}
We now study the term related to $I_l^2(f,g,t)(x)$. Since $x\in B_t$,
$$ \iint\limits_{C_l(B_t \times B_t)}    \left(1+\frac{|x-y|}{t^{\frac{1}{m}}} \right)^{-2n-\varepsilon} \left(1+\frac{|x-z|}{t^{\frac{1}{m}}} \right)^{-2n-\varepsilon} \frac{dydz}{t^{2n/m}} \lesssim 2^{-l(\varepsilon+n)}.$$
Integrating in the variable $t\in(0,r(B)^m)$, we obtain
 \begin{align*}
  \int_0^{r(B)^m} I_l^2(f,g,t)(x) \, \frac{dt}{t}  & \\
 & \hspace{-3cm} \lesssim \int_0^{r(B)^m} 2^{-l(\varepsilon+n)} \sum_{k=0}^{l-1} \iint_{2^{k+1} B_t \times 2^{k+1} B_t}  \frac{|\nabla f(a)||g(b)| + |f(a)||\nabla g(b)|}{(|x-a|+|x-b|)^{2n-1}}  dadb \, \frac{dt}{t} \\
 & \hspace{-3cm} \lesssim l 2^{-l(\varepsilon+n)}  \iint_{2^lB \times 2^l B} \left(\int_{\genfrac{}{}{0pt}{}{0\leq t\leq r(B)^m}{\genfrac{}{}{0pt}{}{|x-a|\leq 2^lt^{1/m}}{|x-b|\leq 2^l t^{1/m}}}} \, \frac{dt}{t} \right)\frac{|\nabla f(a)||g(b)| + |f(a)||\nabla g(b)|}{(|x-a|+|x-b|)^{2n-1}} dadb \\
&  \hspace{-3cm} \lesssim l 2^{-l(\varepsilon+n)}  \iint_{2^lB \times 2^l B} \frac{|\nabla f(a)||g(b)| + |f(a)||\nabla g(b)|}{(|x-a|+|x-b|)^{2n-1}}\\
& \times \log\left(1+\frac{r(B)}{2^{-l}(|x-a|+|x-b|)} \right) dadb \\
& \hspace{-3cm} \lesssim l 2^{-l(\varepsilon+n)}  \iint_{2^lB \times 2^l B} \frac{|\nabla f(a)||g(b)| + |f(a)||\nabla g(b)|}{(|x-a|+|x-b|)^{2n-1}} \log\left(\frac{8\cdot2^{l+1} r(B)}{(|x-a|+|x-b|)} \right) dadb \\
& \hspace{-3cm} \lesssim l 2^{-l(\varepsilon+n)}  \Big[\mathcal{J}_{2^{l+1}B}(|\nabla f|\chi_{2^{l+1}B},|g|\chi_{2^{l+1}B})(x)+ {\mathcal J}_{2^{l+1}B}(|f|\chi_{2^{l+1} B},|\nabla g|\chi_{2^{l+1}B})(x) \Big].\end{align*}
Having obtained pointwise estimates both for $I_0(f,g,t)$ and $I_l(f,g,t)$, we can now conclude the proof of the theorem.

\subsection*{End of the proof of Theorem \ref{thm:representation}}

Using the estimates for $I_0(f,g,t)$, $I_l^1(f,g,t)$ and $I_l^2(f,g,t)$, we finally obtain that
\begin{align*}
 & \left|f(x) g(x) - S_{r(B)^m}(f)(x) S_{r(B)^m}(g)(x)\right| \\
& \lesssim  \sum_{l\geq 0} 2^{-l\varepsilon} (1+ l 2^{-ln})\\
& \times \Big[\mathcal{J}_{2^{l+1}B}(|\nabla f|\chi_{2^{l+1}B},|g|\chi_{2^{l+1}B})(x)+ {\mathcal J}_{2^{l+1}B}(|f|\chi_{2^{l+1} B},|\nabla g|\chi_{2^{l+1}B})(x)\Big] \\
& \lesssim \sum_{l\geq 0} 2^{-l\varepsilon} \Big[\mathcal{J}_{2^{l+1}B}(|\nabla f|\chi_{2^{l+1}B},|g|\chi_{2^{l+1}B})(x)+ {\mathcal J}_{2^{l+1}B}(|f|\chi_{2^{l+1} B},|\nabla g|\chi_{2^{l+1}B})(x)\Big].
\end{align*}
\end{proof}

\subsection{Boundedness properties of the operator ${\mathcal J}_B$}\label{sec:operatorJ}

Boundedness properties of the operators ${\mathcal J}_B$ follow from results for multilinear potential operators in the context of spaces of homogeneous type studied in \cite{MMN}. We use those results, which were recalled in Section \ref{sec:tools}, to prove the following proposition.

\begin{prop}\label{jbbound}  Let $p_1,p_2>1,$  $\,q>0,$ $0<\alpha\le 1$ and
$\frac{1}{q}=\frac{1}{p_1}+\frac{1}{p_2} - \frac{1-\alpha}{n}.$ If  $(w_1,w_2)$ belongs to the class $A_{(p_1,p_2),q}$ then  the operator $\mathcal{J}_B$ defined in \eqref{eq:J} satisfies
$$ \left\|{\mathcal J}_{B} \right\|_{L^{p_1}_{w_1}(B) \times L^{p_2}_{w_2}(B) \to L^q_{w}(B)} \lesssim \left[r(B)\right]^\alpha,$$
with a constant uniform in $B.$
\end{prop}
\begin{proof} Following the results in \cite{MMN}, we work in the space of homogeneous type $(B,\abs{\cdot -\cdot}, dx)$ noting that the constants in  \eqref{quasi}, \eqref{doublingMu}, \eqref{reversedoubling} are independent of $B$.

We will consider the kernel
\[
\tilde{K}((x,y),(a,b)):= \frac{1}{(\abs{x-a}+\abs{y-b})^{2n-1}}\log\left(\frac{8\,r(B)}{\abs{x-a}+\abs{y-b}}\right), \qquad  x,\,y,\,a,\,b\in B
\]
and check that $\tilde{K}$ satisfies \eqref{growth} and \eqref{phiassumption}. For condition \eqref{growth}, note  that for any $c>1$ the function $h(t)=\frac{1}{t^{2n-1}}\log(\frac{8r(B)}{t})$ satisfies $h(t)\le C h(t')$ if $t'\le c\,t$ and $t,\,t'\le 4r(B),$ for some $C>0$ independent of $B.$
Regarding condition \eqref{phiassumption}, recall that the ball with center $x\in B$ and radius $r>0$ in the space $(B,\abs{\cdot -\cdot}, dx)$ is $B(x,r) \cap B$ where $B(x,r)$ is the Euclidean ball in $\re^n$ of radius $r$ centered at $x.$  Since  for $x\in B$ and $r\lesssim r(B),$ $\abs{B(x,r)\cap B}\sim \abs{B(x,r)}=c_n\, r^n,$ we then have to prove  that there exists $\delta>0$ such that given $C_1>1$ there is $C_2>0$  independent of $B$ for which
\[
\frac{\varphi(B_1 \cap B)}{\varphi(B_2\cap B)}\le C_2\, \left(\frac{r_2}{r_1}\right)^{2n-\delta}, \]
 for all balls  $B_i=B(x_i,r_i),$ $x_i\in B,$  $r_i\le C_1r(B),$ $ B_1\cap B\subset B_2\cap B$,
where
\[
\varphi(B_i\cap B)=\sup\{K(x,a,b): x,\,a,\,b \in B_i\cap B, \, \abs{x-a}+\abs{x-b}\ge c\, r_i\}
\]
for some fixed positive small constant $c$ and $i=1,\,2$.
We have  $\varphi(B_i\cap B)=(\frac{1}{cr_i})^{2n-1}\log(\frac{8r(B)}{cr_i})$  which gives
\[
\frac{\varphi(B_1\cap B)}{\varphi(B_2\cap B)}\sim\left(\frac{r_2}{r_1}\right)^{2n-1} \, \frac{\log\left(\frac{8r(B)}{cr_1}\right)}{\log\left(\frac{8r(B)}{cr_2}\right)}\lesssim \left(\frac{r_2}{r_1}\right)^{2n-\delta}, \quad 0<\delta<1,
\]
since $\frac{\log(t')}{\log(t)}\lesssim (\frac{t'}{t})^{\gamma}$ for $2\le t\le t'$ and $0<\gamma<1.$

We now check that the assumptions on the weights $w_1,$ $w_2,$ and $w$ imply  \eqref{general2wq>1hm} if $q>1$ and \eqref{general2wq<1hm} if $q\le1$ with $u=w^{1/q}=w_1^{1/p_1}w_2^{1/p_2},$  $v_k=w_k^{1/p_k},$ $k=1,2.$ This means that we have to prove  that there exists $t>1$ such that
\begin{equation}\label{qlarger1}
\sup_{Q} \varphi(Q) |Q|^{\frac{1}{q} + \frac{1}{{p_1}'}+\frac{1}{{p_2'}}}
\left( \frac{1}{|Q|}\int_Q w^{t} dx \right)^{1/qt}
\prod_{j=1}^2 \left( \frac{1}{|Q|}\int_Q w_i^{-\frac{t}{p_i-1}} dx
\right)^{1/tp_i'} < \infty, \quad q>1,
\end{equation}
and
\begin{equation}\label{qsmaller1}
\sup_{Q} \varphi(Q) |Q|^{\frac{1}{q} + \frac{1}{{p_1}'}+\frac{1}{{p_2'}}}
\left( \frac{1}{|Q|}\int_Q w dx \right)^{1/q}
\prod_{j=1}^2 \left( \frac{1}{|Q|}\int_Q w_i^{-\frac{t}{p_i-1}} dx
\right)^{1/tp_i'} < \infty, \quad q\le1,
\end{equation}
where the sup is taken over all balls $Q$ in the space $(B,|\cdot-\cdot|,dx)$ with $r(Q)\lesssim r(B).$ The proofs follow using the same ideas  as in  Remark \ref{weightscomparison} .
Let $Q$ be a ball in the space $(B,|\cdot-\cdot|,dx)$  with $r(Q)\lesssim r(B);$ then $Q=B\cap B(x,r)$ for some $x\in B$ and $r>0,$  $r(Q)=r\lesssim r(B)$ and $|Q|\sim |B(x,r)|.$
Moreover, using the relation between $p_1,\,p_2,\,q$ and $\alpha$ as in the statement of the proposition,
\begin{align*}
\varphi(Q)|Q|^{\frac{1}{q} + \frac{1}{{p_1}'}+\frac{1}{{p_2'}}} & \sim \frac{1}{r(Q)^{2n-1}}\log\left(\frac{8r(B)}{c\,r(Q)}\right)\,r(Q)^{2n-1+\alpha}\\
&= r(Q)^\alpha\,\log\left(\frac{8r(B)}{c\,r(Q)}\right)\lesssim r(B)^\alpha.
\end{align*}
In addition,  the second factor in \eqref{qlarger1} is bounded  by
\begin{align}\label{qlarger1_2}
\sup_{x\in\re^n,r>0}
\left( \frac{1}{|B(x,r)|}\int_{B(x,r)} w^{t} dx \right)^{1/qt}
\prod_{j=1}^2 \left( \frac{1}{|B(x,r)|}\int_{B(x,r)} w_i^{-\frac{t}{p_i-1}} dx
\right)^{1/tp_i'}.
\end{align}
Since  $w,\,w_1^{-\frac{1}{p_1-1}},\,w_2^{-\frac{1}{p_2-1}}$ are $A_\infty$ weights (see Remark \ref{rem:weights}), there exists $t>1$ such that \eqref{qlarger1_2} is bounded by
\[
\sup_{x\in\re^n,r>0}
\left( \frac{1}{|B(x,r)|}\int_{B(x,r)} w \,dx \right)^{1/q}
\prod_{j=1}^2 \left( \frac{1}{|B(x,r)|}\int_{B(x,r)} w_i^{-\frac{1}{p_i-1}} dx
\right)^{1/p_i'}<\infty,
\]
where finiteness is due to $(\,w_1,w_2)$ satisfying the $A_{(p_1,p_2),q}$ condition. A similar reasoning applies to \eqref{qsmaller1}.
We conclude that \eqref{qlarger1} and \eqref{qsmaller1} are bounded by a multiple (independent of $B$) of $r(B)^\alpha.$

By Theorem~\ref{general2w} and Remark~\ref{remm} we have that $\mathcal{J}_B$ is bounded from $L^{p_1}_{w_1}(B)\times L^{p_2}_{w_2}(B)$ into $L^q_{w}(B)$  and the operator norm  is bounded by  a multiple (uniform on $B$) of $ r(B)^\alpha.$
\end{proof}

\subsection{Proof of Theorem \ref{thm:bp}}
 Let  $p_1,\,p_2,$ $q,$ $w_1,$ $w_2$ and $w$  be as in the statement of  Theorem~\ref{thm:bp}.
By  Proposition \ref{jbbound} we have
$$ \left\|{\mathcal J}_{2^l B} \right\|_{L^{p_1}_{w_1}(B) \times L^{p_2}_{w_2}(B) \to L^q_{w}(B)} \lesssim \left[2^l r(B)\right]^\alpha,$$
uniformly in $B$ and $l\geq 0$, this and Theorem \ref{thm:representation} imply
\begin{align*}
\lefteqn{\left\|f g - S_{r(B)^m}(f) S_{r(B)^m}(g)\right\|_{L^q_w(B)}} & \\
&   \lesssim \sum_{l\geq 0} 2^{-l\varepsilon} 2^{\alpha (l+1)} r(B)^\alpha \left[\left\|\nabla f \right\|_{L^{p_1}_{w_1}(2^{l+1}B)} \left\|g \right\|_{L^{p_2}_{w_2}(2^{l+1}B)} + \left\|f\right\|_{L^{p_1}_{w_1}(2^{l+1} B)} \left\|\nabla g\right\|_{L^{p_2}_{w_2}(2^{l+1}B)} \right],
\end{align*}
which concludes the proof of Theorem~\ref{thm:bp}.
\qed

\bigskip

Applying an analogous  proof to that of Theorem \ref{thm:bp}, we obtain the following result:

\begin{theor} \label{coro:bp} Under the same assumptions of Theorem
\ref{thm:bp},
\begin{align*}
\lefteqn{\left\|f g - S_{r(B)^m} \left[S_{r(B)^m}(f)
S_{r(B)^m}(g)\right]\right\|_{L^q_w(B)}} & & \\
& &  \le C\, r(B)^\alpha \sum_{l\geq 0} 2^{-l(\varepsilon-\alpha)}
\left[\left\|\nabla f \right\|_{L_{w_1}^{p_1}(2^{l+1}B)} \left\|g
\right\|_{L^{p_2}_{w_2}(2^{l+1}B)} +
\left\|f\right\|_{L^{p_1}_{w_1}(2^{l+1} B)} \left\|\nabla
g\right\|_{L^{p_2}_{w_2}(2^{l+1}B)} \right].
\end{align*}
  \end{theor}
We will leave it to the reader  to check the details for the fact that the proof of Theorem \ref{thm:bp} still
holds after noting that   a similar representation formula can be used as we can write
$$ f g - S_{r(B)^m} \left[S_{r(B)^m}(f) S_{r(B)^m}(g)\right] = -\int_0^{r(B)^m} t \partial_t S_{t}\left[\mathcal{P}_t(F) \right]
\frac{dt}{t},$$
since $t \partial_t S_{t}\left[\mathcal{P}_t \right]$ satisfies the
same estimates as $t \partial_t \mathcal{P}_t$ and the cancellation
property $ t \partial_t S_{t}\left[\mathcal{P}_t ({\bf 1})\right] = 0.$

\section{Leibniz-type rules in Campanato-Morrey spaces associated to generalized approximations of identity} \label{sec:campanato}

In this section we apply Theorem \ref{thm:bp} to prove a Leibniz-type rule of the form \eqref{leibrule} where the spaces $X_1$, $X_2$, $Y_1$, $Y_2$ belong to the scale of the classical Campanato-Morrey spaces and the space $Z$ quantifies the oscillation $|fg - S_{r(B)^m}(f)S_{r(B)^m}(f)|$ of the product $fg$ in $L^q(B)$ where $B \subset \re^n$ is a Euclidean  ball in $\re^n$ (compare to  \eqref{leibcampanatomorrey}).
In this context, it will become clear how, as announced in the Introduction, the bilinear potential operators introduced in Section \ref{sec:tools} play the role that paraproducts and the bilinear Coifman-Meyer multipliers play in the proofs of the Sobolev-based Leibniz-type rules \eqref{katoponce}.

Next, we recall the definition of the classical Campanato-Morrey spaces and introduce notions of bilinear Campanato-Morrey spaces associated to approximations of the identity and  semigroups.

For $p >0$ and $\lambda \geq 0$ we say that $f \in L^1_{loc}(\re^n)$ belongs to the \emph{Campanato-Morrey space} $\mathcal{L}^{p,\lambda}(\re^n)$ if
\begin{equation}
\norm{f}{\mathcal{L}^{p,\lambda}(\re^n)}:=\sup\limits_{B \subset \re^n}\frac{1}{|B|^{\lambda}} \left(\frac{1}{|B|} \int_B |f(x)|^p \, dx \right)^{\frac{1}{p}}
\end{equation}
is finite. For $f, g \in L^1(\re^n)$ we say that the pair $(f,g)$ belongs to the \emph{bilinear Campanato-Morrey space} $L_{{\mathcal S} \otimes  {\mathcal S} }^{p,\lambda}(\re^n)$ associated to an approximation of the identity  ${\mathcal S}=\{S_t\}_{t>0}$ of order $m>0$ if
\begin{equation}\label{biCM}
\norm{(f,g)}{L_{\mathcal{S} \otimes \mathcal{S}}^{p,\lambda}(\re^n)}: = \sup\limits_{B \subset \re^n} \frac{1}{|B|^{\lambda}} \left(\frac{1}{|B|} \int_B |f(x)g(x) - S_{r(B)^m}(f)(x)S_{r(B)^m}(g)(x)|^p \, dx \right)^{\frac{1}{p}}
\end{equation}
is finite. We use the notation ${\mathcal S} \otimes {\mathcal S}$ to signify that the oscillation in question coincides with the tensorial oscillation $|(f \otimes g)(x,y) - (S \otimes S)_{r(B)^m} (f \otimes g) (x,y)|$, for $x, y \in B$, restricted to the diagonal $x=y$. These new spaces $L_{{\mathcal S} \otimes {\mathcal S}}^{p,\lambda}(\re^n)$ arise as  natural bilinear counterparts to the Campanato-Morrey spaces $L_{\mathcal{S}}^{p,\lambda}(\re^n)$ associated to $\mathcal{S}$ introduced by Duong and Yan in \cite{DY1, DY2}. In this case, $f \in L_{\mathcal{S}}^{p,\lambda}(\re^n)$ if
\begin{equation}\label{CMPt}
\norm{f}{L_{\mathcal{S}}^{p,\lambda}(\re^n)}: = \sup\limits_{B \subset \re^n} \frac{1}{|B|^{\lambda}} \left(\frac{1}{|B|} \int_B |f(x) - S_{r(B)^m}(f)(x)|^p \, dx \right)^{\frac{1}{p}} < \infty.
\end{equation}

\begin{theor}\label{leibniz2} Let ${\mathcal S}:=\{S_t\}_{t>0}$ and ${\mathcal S'}:=\{t\partial_t S_t\}_{t>0}$ be   approximations of the identity of order $m>0$ in $\re^n$  and constant $\varepsilon$ in \eqref{gammadecay},  $1<p_1,p_2<\infty,$  $0<\alpha<\min(\varepsilon,1)$ and $q>0$ such that
$\frac{1}{q}=\frac{1}{p_1}+\frac{1}{p_2} - \frac{1-\alpha}{n}.$
Given $\lambda_1, \lambda_2 \geq 0$ set
$
\lambda = \frac{1}{n}+ \lambda_1 + \lambda_2
$
and assume that
$
\varepsilon > n \left(\lambda + \frac{1}{q} \right).
$
Then there exists a structural constant $C > 0$ such that the following Leibniz-type rule holds true
\begin{equation}\label{leibCMPt}
\norm{(f,g)}{L_{{\mathcal S} \otimes {\mathcal S}}^{q,\lambda}(\re^n)} \leq C \left(\norm{\nabla f}{\mathcal{L}^{p_1,\lambda_1}(\re^n)} \norm{g}{\mathcal{L}^{p_2,\lambda_2}(\re^n)} + \norm{f}{\mathcal{L}^{p_1,\lambda_1}(\re^n)} \norm{\nabla g}{\mathcal{L}^{p_2,\lambda_2}(\re^n)}\right).
\end{equation}

\end{theor}

\begin{proof}  From Theorem \ref{thm:bp} we have
\begin{align*}
\lefteqn{\left\|f g - S_{r(B)^m}(f) S_{r(B)^m}(g)\right\|_{L^q(B)}} & & \\
& &  \lesssim r(B)^\alpha \sum_{l\geq 0} 2^{-l(\varepsilon-\alpha)} \,\left( \left\|\nabla f \right\|_{L^{p_1}(2^{l}B)} \left\|g \right\|_{L^{p_2}(2^{l}B)} +  \left\|f\right\|_{L^{p_1}(2^{l} B)} \left\|\nabla g\right\|_{L^{p_2}(2^{l}B)} \right).
\end{align*}
By writing
$$
 \left\|\nabla f \right\|_{L^{p_1}(2^{l}B)} = |2^l B|^{\lambda_1 + \frac{1}{p_1}} \frac{1}{|2^l B|^{\lambda_1}}\left( \frac{1}{|2^l B|} \int_{2^l B} |\nabla f|^{p_1} \right)^{\frac{1}{p_1}} \leq |2^l B|^{\lambda_1 + \frac{1}{p_1}} \norm{\nabla f}{\mathcal{L}^{p_1,\lambda_1}(\re^n)}
$$
and
$$
 \left\|g \right\|_{L^{p_2}(2^{l}B)} = |2^l B|^{\lambda_2 + \frac{1}{p_2}} \frac{1}{|2^l B|^{\lambda_2}}\left( \frac{1}{|2^l B|} \int_{2^l B} |g|^{p_2} \right)^{\frac{1}{p_2}} \leq |2^l B|^{\lambda_2 + \frac{1}{p_2}} \norm{g}{\mathcal{L}^{p_2,\lambda_2}(\re^n)},
$$
and similarly with $\left\|f \right\|_{L^{p_1}(2^{l}B)}$ and $\left\|\nabla g\right\|_{L^{p_2}(2^{l}B)},$ and by setting $s:=\lambda_1 + \lambda_2 + \frac{1}{p_1} + \frac{1}{p_2}$, we obtain
\begin{align*}
& r(B)^\alpha \sum_{l\geq 0} 2^{-l(\varepsilon-\alpha)} \,\left( \left\|\nabla f \right\|_{L^{p_1}(2^{l}B)} \left\|g \right\|_{L^{p_2}(2^{l}B)} +  \left\|f\right\|_{L^{p_1}(2^{l} B)} \left\|\nabla g\right\|_{L^{p_2}(2^{l}B)} \right) \\
& \leq |B|^{\frac{\alpha}{n}+ s} \sum_{l\geq 0} 2^{-l(\varepsilon-\alpha - ns)} \,\left( \norm{\nabla f}{\mathcal{L}^{p_1,\lambda_1}(\re^n)} \norm{g}{\mathcal{L}^{p_2,\lambda_2}(\re^n)} + \norm{ f}{\mathcal{L}^{p_1,\lambda_1}(\re^n)}\norm{\nabla g}{\mathcal{L}^{p_2,\lambda_2}(\re^n)} \right) \\
& \leq C |B|^{\frac{\alpha}{n}+ s}  \left( \norm{\nabla f}{\mathcal{L}^{p_1,\lambda_1}(\re^n)} \norm{g}{\mathcal{L}^{p_2,\lambda_2}(\re^n)} + \norm{ f}{\mathcal{L}^{p_1,\lambda_1}(\re^n)}\norm{\nabla g}{\mathcal{L}^{p_2,\lambda_2}(\re^n)} \right),
\end{align*}
since $\alpha + ns = n \left(  \frac{\alpha}{n}+ \frac{1}{p_1} + \frac{1}{p_2} + \lambda_1 + \lambda_2 \right) = n \left( \lambda + \frac{1}{q}\right) < \varepsilon$.
Consequently, using that $\lambda + \frac{1}{q} = \frac{\alpha}{n}+ s$,
\begin{align*}
 & \hspace{-2cm} \frac{1}{|B|^{\lambda}} \left(\frac{1}{|B|} \int_B |f(x)g(x) - S_{r(B)^m}(f)(x)S_{r(B)^m}(g)(x)|^q \, dx \right)^{\frac{1}{q}}  \\
& = \frac{1}{|B|^{\frac{\alpha}{n}+ s}}\left\|f g - S_{r(B)^m}(f) S_{r(B)^m}(g)\right\|_{L^q(B)}\\
 & \leq C   \left( \norm{\nabla f}{\mathcal{L}^{p_1,\lambda_1}(\re^n)} \norm{g}{\mathcal{L}^{p_2,\lambda_2}(\re^n)} + \norm{ f}{\mathcal{L}^{p_1,\lambda_1}(\re^n)}\norm{\nabla g}{\mathcal{L}^{p_2,\lambda_2}(\re^n)} \right),
\end{align*}
and \eqref{leibCMPt} follows.
\end{proof}

In relation with \eqref{CMPt}, we define another  suitable notion of Campanato-Morrey spaces associated
to an approximation of the identity ${\mathcal S}=\{S_t\}$: a function $f$ belongs to $\tilde{L}_{\mathcal
S}^{p,\lambda}(\re^n)$ if
$$ \|f\|_{\tilde{L}_{\mathcal S}^{p,\lambda}(\re^n)} := \sup\limits_{B \subset
\re^n} \inf_{h\in L^1_{loc}} \frac{1}{|B|^{\lambda}} \left(\frac{1}{|B|}
\int_B |f(x) - S_{r(B)^m}(h)(x)|^p \, dx \right)^{\frac{1}{p}} < \infty,$$
where  the supremum is taken over all Euclidean balls $B\subset \re^n.$
Then, we have the following Leibniz-type rule:

\begin{theor}\label{leibniz3} Let ${\mathcal S}:=\{S_t\}_{t>0}$ and ${\mathcal
S'}:=\{t\partial_t S_t\}_{t>0}$ be   approximations of the identity in
$\re^n$ of order $m>0$ and constant $\varepsilon$ in \eqref{gammadecay},  $1<p_1,p_2<\infty,$
$0<\alpha<\min(\varepsilon,1)$ and $q>0$ such that
$\frac{1}{q}=\frac{1}{p_1}+\frac{1}{p_2} - \frac{1-\alpha}{n}.$
Given $\lambda_1, \lambda_2 \geq 0$ set
$
\lambda = \frac{1}{n}+ \lambda_1 + \lambda_2
$
and assume that
$
\varepsilon > n \left(\lambda + \frac{1}{q} \right).
$
Then there exists a structural constant $C > 0$ such that the following
Leibniz-type rule holds true
\begin{equation}\label{leibCMPt-bis}
\norm{fg}{\tilde{L}_{{\mathcal S}}^{q,\lambda}(\re^n)} \leq C \left(\norm{\nabla
f}{\mathcal{L}^{p_1,\lambda_1}(\re^n)}
\norm{g}{\mathcal{L}^{p_2,\lambda_2}(\re^n)} +
\norm{f}{\mathcal{L}^{p_1,\lambda_1}(\re^n)} \norm{\nabla
g}{\mathcal{L}^{p_2,\lambda_2}(\re^n)}\right).
\end{equation}
\end{theor}

The proof follows by estimating the norm
$$ \sup\limits_{B \subset \re^n} \inf_{h\in L^1_{loc}}
\frac{1}{|B|^{\lambda}} \left(\frac{1}{|B|} \int_B |f(x) g(x) -
S_{r(B)^m}(h)(x)|^p \, dx \right)^{\frac{1}{p}}$$
with $h=S_{r(B)^m}(f) S_{r(B)^m}(g)$  and following the
arguments in Theorem~\ref{leibniz2} by invoking   Theorem \ref{coro:bp} instead of Theorem
\ref{thm:bp}.

\section{Extensions to  doubling Riemannian manifolds and Carnot groups} \label{sec:extensions}

\subsection{Doubling Riemannian manifolds}\label{sec:dm}

Let $(M,\rho,d\mu)$ be a doubling Riemannian manifold, this is  a space of homogeneous type with a gradient vector field $\nabla$ (e.g. a complete Riemannian manifold with nonnegative Ricci curvature).

An approximation of the identity  of order $m>0$  in $M$ is a collection of operators ${\mathcal S}:=\{S_t\}_{t>0}$ acting on functions defined on $M,$
\[
S_tf(x)=\int_{M}p_t(x,y)f(y)\,d\mu(y),
\]
such that for each $t>0$ the kernels $p_t$ satisfy $\int_{M}p_t(x,y)\,d\mu(y)=1$ for all $x$  and  the \emph{scaled Poisson bound}
\begin{equation}\label{pboundho}
\abs{p_t(x,y)}\le \mu(B_\rho(x,t^{1/m}))^{-1}\,\gamma\left(\frac{\rho(x,y)}{t^{1/m}}\right),
\end{equation}
where $\gamma:[0,\infty)\to[0,\infty)$ is a bounded, decreasing function such that
 \begin{equation}\label{gammadecayho}
\lim_{r\to\infty} r^{2n+\varepsilon}\gamma(r)=0,\qquad \text{ for some }\, \varepsilon>0.
\end{equation}

\begin{theor} \label{thm:bpmanifold}  Assume $(M,\rho,\mu)$ is a doubling Riemannian manifold.
Let ${\mathcal S}:=\{S_t\}_{t>0}$  and ${\mathcal S'}:=\{t\partial_t S_t\}_{t>0}$ be approximations of the identity in $M$ of order $m>0$ and constant $\varepsilon$ in \eqref{gammadecayho},
 $1<p_1,p_2<\infty,$ $q>0,$ and  $0<\alpha <\min\{1,\varepsilon\}$ such that
$\frac{1}{q}=\frac{1}{p_1}+\frac{1}{p_2} - \frac{1-\alpha}{n}.$
Then  there exists a  constant $C$ such that for all  balls $B\subset M$
\begin{align*}
\lefteqn{\left\|f g - S_{r(B)^m}(f) S_{r(B)^m}(g)\right\|_{L^q(B)}} & & \\
& &  \le C\, r(B)^\alpha \sum_{l\geq 0} 2^{-l(\varepsilon-\alpha)} \left[\left\|\nabla f \right\|_{L^{p_1}(2^{l+1}B)} \left\|g \right\|_{L^{p_2}(2^{l+1}B)} + \left\|f\right\|_{L^{p_1}(2^{l+1} B)} \left\|\nabla g\right\|_{L^{p_2}(2^{l+1}B)} \right].
\end{align*}
\end{theor}

The proof of this theorem follows from that of Theorem \ref{thm:bp} after minor modifications. The Leibniz rules in Campanato/Morrey spaces, obtained in Section \ref{sec:campanato}, can be extended to this framework as well.

\subsection{Carnot groups}\label{sec:carnot}

In this section we provide a description of how to extend our results of section \ref{sec:bpapproxident} in the context of  Carnot groups.
  Let $\Omega$ be an open connected subset of $\re^n$ and
${\bf X}=\{X_k\}_{k=1}^M$ be a family of infinitely
differentiable vector fields with values in $
\re^n$. We identify $X_k$ with the first order
differential operator acting on continuously differentiable functions defined on
$\Omega$ by the formula
\[
X_kf(x)=X_k(x)\cdot \nabla f(x), \quad k=1,\cdots,M,
\]
and we set ${\bf X}f=(X_1 f, X_2f,\cdots, X_Mf)$ and
\[
\abs{{\bf X}f(x)}=\left(\sum_{k=1}^M\abs{X_k f(x)}^2\right)^{1/2}, \quad
x\in \Omega.
\]
Given two vector fields $X_i$ and $X_j$ define  the commutator or Lie bracket by $[X_i,X_j]=X_iX_j-X_jX_i$.  We will assume that  ${\bf X}$ satisfies H\"ormander's condition in $\Omega$; that is, there is some finite positive integer $M_0$ such that the commutators of the vector fields in ${\bf X}$ up to length $M_0$ span $\re^n$ at each point of $\Omega$.

Suppose that ${\bf X}=\{X_k\}_{k=1}^M$ satisfies H\"ormander's
condition in $\Omega.$  Let $C_{\bf X}$ be the family of absolutely
continuous curves $\zeta:[a,b]\to\Omega,$ $a\le b,$ such that
there exist measurable functions $c_j(t),$ $a\le t\le b,$
$j=1,\cdots, M,$ satisfying $\sum_{j=1}^{M} c_j(t)^2\le 1$ and
$\zeta'(t)=\sum_{j=1}^{M}c_j(t) Y_j(\zeta(t))$ for almost every
$t\in [a,b].$ If $x,\,y\in\Omega$ define
\[\rho(x,y)=\inf\{T>0:\text{ there exists }\zeta \in C_{\bf X} \text{ with } \zeta(0)=x \text{ and }\zeta(T)=y\}.\]
The function $\rho$ is in fact a metric in $\Omega$ called the
Carnot-Carath\'eodory metric on $\Omega$ associated to ${\bf X}$.

Let $\mathbb{G}$ be a Lie group on $\re^n$, that is a group law on $\re^n$ such that the map $(x,y)\mapsto xy^{-1}$ is $C^\infty$.  The Lie algebra associated to $\mathbb{G}$, denoted $\mathfrak{g}$, is the collection of all left invariant vector fields on $\mathbb{G}$.  A {\it Carnot group} is a Lie group whose Lie algebra admits a stratification
$$\mathfrak{g}=V_1\oplus\cdots \oplus V_l,$$
where $[V_1,V_i]=\text{span}\{[Y,Z]: Y\in V_1,Z\in V_i\}=V_{i+1},$ $i=1,\cdots,l-1$, and $[V_1,V_i]=\{0\}$ for $i\ge l$.  A basis for $V_1$ generates the whole Lie algebra.  We will often denote this family  as $\{X_1,\ldots, X_{n_1}\}$ and refer to it as a family of generators for the Carnot group.  In particular, a system of generators $\{X_1,\ldots, X_{n_1}\}$ satisfies H\"ormander's condition, and hence we have the notion of a Carnot-Caratheodory metric.

Set $n_i=\text{dim}(V_i)$, then $n=n_1+\cdots +n_l$, and the number $Q=\sum_{i=1}^l in_i$ is called the homogeneous dimension of $\mathbb{G}$ .  The dilation operators
$$\delta_\lambda x=(\lambda x^{(1)},\lambda^2 x^{(2)},\ldots, \lambda^l x^{(l)}) \qquad x^{(i)}\in \re^{n_i}$$
form automorphisms of $\mathbb{G}$ for each $\lambda>0$.  Furthermore, if $B$ is a metric ball of radius $r(B)$ with respect to the Carnot-Carath\'eodory metric then $|B|=c\, r(B)^Q$, which shows that $(\re^n,\rho,\text{Lebesgue measure})$ is a space of homogenous type.  We refer the reader to \cite{BLU} for more information about analysis on Carnot groups.

An approximation of the identity  of order $m>0$ in $\mathbb{G}$ is a collection of operators ${\mathcal S}:=\{S_t\}_{t>0}$ acting on functions defined on $\re^n,$
\[
S_tf(x)=\int_{\re^n}p_t(x,y)f(y)\,dy,
\]
such that for each $t>0$ the kernels $p_t$ satisfy $\int_{\re^n}p_t(x,y)\,dy=1$ for all $x$  and  the \emph{scaled Poisson bound}
\begin{equation*}
|p_t(x,y)|\leq t^{-Q/m}\gamma\Big(\frac{\rho(x,y)}{t^{1/m}}\Big),
\end{equation*}
where $\gamma:[0,\infty)\to[0,\infty)$ is a bounded, decreasing function such that
 \begin{equation*}
\lim_{r\to\infty} r^{2Q+\varepsilon}\gamma(r)=0, \quad \text{for some} \ \varepsilon>0.
\end{equation*}

\begin{theor} \label{thm:bpcar} Suppose $\mathbb{G}$ is a homogeneous Carnot group of dimension $Q$ with generators ${\bf{X}}=\{X_1,\ldots, X_{n_1}\}$ and $\rho$ is the Carnot-Carath\'eodory metric on $\re^n$ associated to $\bf{X}$.  Suppose further that $\mathcal{S}=\{S_t\}_{t>0}$ and $\mathcal{S}'=\{t\partial_t S_t\}_{t>0}$ are approximations of the identity in $\mathbb{G}$ of order $m$ and $\varepsilon$ as given above. If $p_1,p_2>1,$  $0<\alpha<\min(\varepsilon,1)$ and $q>0$ are such that
$\frac{1}{q}=\frac{1}{p_1}+\frac{1}{p_2} - \frac{1-\alpha}{Q},$
then, for every  $\rho$-ball $B,$
\begin{align*}
\lefteqn{\left\|f g - S_{r(B)^m}(f) S_{r(B)^m}(g)\right\|_{L^q(B)}} & & \\
& &  \lesssim r(B)^\alpha \sum_{l\geq 0} 2^{-l(\varepsilon-\alpha)} (l+1)\,\left[\left\|{\bf X} f \right\|_{L^{p_1}(2^{l+1}B)} \left\|g \right\|_{L^{p_2}(2^{l+1}B)}+ \left\|f\right\|_{L^{p_1}(2^{l+1} B)} \left\| {\bf X} g\right\|_{L^{p_2}(2^{l+1}B)} \right].
\end{align*}
\end{theor}

\begin{proof}[Sketch of Proof] We will take the same approach as the proof of Theorem \ref{thm:bp}.  The multilinear representation formula is given by
\begin{align}
|f(x)g(x&)-S_{r(B)^m} f(x)S_{r(B)^m}g(x)|\nonumber\\
\label{repsum}&\lesssim \sum_{l\geq 0} 2^{-l(\varepsilon-Q)}[{\mathcal J}_{2^{l+1}B}(|{\bf X}f|,|g|)(x)+{\mathcal J}_{2^{l+1}B}(|f|,|{\bf X}g|)(x).
\end{align}
where
$${\mathcal J}_{B}(f,g)(x)=\iint_{B\times B}\frac{f(y)g(z)}{(\rho(x,y)+\rho(x,z))^{2Q-1}}\log\Big(\frac{cr(B)}{\rho(x,y)+\rho(x,z)}\Big)\, dydz \qquad x\in B$$ and $B$ is a ball in $\re^n$ with respect to the metric $\rho.$
The operator ${\mathcal J}_{B}$ satisfies the necessary  growth bounds on its kernel and  hence
\begin{equation}\label{Jbound}\|\mathcal J_B\|_{L^{p_1}(B)\times L^{p_2}(B)\rightarrow L^q(B)}\lesssim [r(B)]^\alpha. \end{equation}
The  inequalities \eqref{repsum} and \eqref{Jbound} prove the desired result.  The proof of  inequality \eqref{repsum} follows that  of Theorem \ref{thm:representation} with the Euclidean distance replaced by $\rho(x,y)$ and the dimension $n$ replaced by $Q$.  We just  highlight  the analog to inequality \eqref{eq:aze},
\begin{equation}\label{homog} \iint\limits_{B_t\times B_t}\frac{1}{(\rho(y,a)+\rho(z,b))^{2Q-1}}\,dydz\lesssim \frac{1}{(\rho(x,a)+\rho(x,b))^{2Q-1}}.\end{equation}

Let $B=B_\rho$ be a ball in $\re^n$ with respect to the metric $\rho$, $x\in B$, $r(B)$ be the radius of $B$.  Suppose $0<t<r(B)^m$ and $a,b\in B_t=B_\rho(x,t^{1/m})$ then
\begin{align*}
\iint\limits_{B_t\times B_t}\frac{1}{(\rho(y,a)+\rho(z,b))^{2Q-1}}\,dydz &\lesssim \iint\limits_{B_\rho(a,2t^{1m})\times B_\rho(b,2t^{1/m})}\frac{1}{(\rho(y,a)+\rho(z,b))^{2Q-1}}\,dydx\\
&\lesssim \sum_{k\geq 0} \ \ \iint\limits_{D_k}\frac{1}{(\rho(y,a)+\rho(z,b))^{2Q-1}}\,dydx
\end{align*}
 where $$D_k:=\{(y,z): 2^{-k}t^{1/m}\leq \rho(a,y)<2^{-k+1}t^{1/m},2^{-k}t^{1/m}\leq \rho(b,z)<2^{-k+1}t^{1/m}\}.$$
We continue estimating each term in the series
\begin{align*}
\lefteqn{\iint\limits_{D_k}\frac{1}{(\rho(y,a)+\rho(z,b))^{2Q-1}}\,dydz}\\
&\lesssim (2^kt^{-1/m})^{2Q-1}|B_\rho(a,2^{-k+1}t^{1/m})|\cdot |B_\rho(b,2^{-k+1}t^{1/m})|\\
&\lesssim 2^{-k}t^{1/m}
\end{align*}
which leads to
\begin{align*}
\iint\limits_{B_t\times B_t}\frac{1}{(\rho(y,a)+\rho(z,b))^{2Q-1}}\,dydz&\lesssim \iint\limits_{B_\rho(a,2t^{1m})\times B_\rho(b,2t^{1/m})}\frac{1}{(\rho(y,a)+\rho(z,b))^{2Q-1}}\,dydx\\
&\lesssim t^{1/m}\\
&\lesssim \frac{1}{(\rho(x,a)+\rho(x,b))^{2Q-1}}
\end{align*}
This estimate contributes to the first term on the right side of inequality \eqref{repsum}, the other terms are obtained in a similar manner.

\end{proof}

\section{Boundedness of bilinear pseudodifferential operators under Sobolev scaling}\label{sec:bpseudo}

Let $BS^m_{\rho,\delta}(\re^n)$ and $BS^m_{\rho,\delta;\theta}(\re^n),$ where $m\in\re,$ $0\le \delta\le \rho\le 1,$ $\theta\in (0,\pi),$ be the classes of symbols $\sigma\in C^{\infty}(\re^{3n})$ satisfying,
\begin{equation} \label{BS1}  \left|  \partial_x^\alpha \partial_\xi^\beta \partial_\eta^\gamma \sigma(x,\xi,\eta) \right| \le\,C_{\alpha,\beta,\gamma}\,\left(1+|\xi|+|\eta|\right)^{m-\rho(|\beta|+|\gamma|)+\delta|\alpha|},
\end{equation}
respectively,
\begin{equation} \label{BS1theta}
  \left|  \partial_x^\alpha \partial_\xi^\beta \partial_\eta^\gamma \sigma(x,\xi,\eta) \right| \le\,C_{\alpha,\beta,\gamma}\,\left(1+|\xi-\tan(\theta)\,\eta|\right)^{m-\rho(|\beta|+|\gamma|)+\delta|\alpha|},
\end{equation}
for all $x,\,\xi,\,\eta\in\re^n,$ all multi-indices $\alpha,\,\beta,\,\gamma\in\na_0^n$ and some constants $C_{\alpha,\beta,\gamma},$ with the convention that  $\theta=\frac{\pi}{2}$ corresponds to decay in terms of $1+\abs{\xi}.$
We will use the notation $\dot{BS^{m}_{1,\delta}}(\re^n)$ and $\dot{BS^{m}_{1,0;\theta}}(\re^n)$  for the homogeneous versions of the above classes,  defined by replacing $1+\abs{\xi}+\abs{\eta}$ by $\abs{\xi}+\abs{\eta}$ and $1+|\xi-\tan(\theta)\,\eta|$ by $|\xi-\tan(\theta)\,\eta|$ in \eqref{BS1} and \eqref{BS1theta}, respectively. Also, we will use $\norm{\sigma}{\alpha, \beta, \gamma}$ to denote the smallest constant $C_{\alpha, \beta,\gamma}$ in \eqref{BS1} or \eqref{BS1theta}.

These classes can be regarded as bilinear counterparts to the linear H\"ormander classes $S^{m}_{\rho,\delta}(\re^n)$ (and their homogeneous analogs $\dot{S^m_{\rho,\delta}}(\re^n)$) which consists of symbols  $\sigma\in C^{\infty}(\re^{2n})$ such that
\begin{equation*}   \left|  \partial_x^\alpha \partial_\xi^\beta \sigma(x,\xi) \right| \le\,C_{\alpha,\beta}\,\left(1+|\xi|\right)^{m-\rho|\beta|+\delta|\alpha|},
\end{equation*}
for all $x,\,\xi\in\re^n,$ all multiindices $\alpha,\,\beta,$ and some constants $C_{\alpha,\beta}.$

Our results in this section assume symbols in the classes $BS^m_{1,\delta}(\re^n)$ or $\dot{BS^m_{1,\delta}}(\re^n)$, as well as those symbols in  $BS^m_{1,\delta;\theta}(\re^n)$ or  $\dot{BS^m_{1,\delta;\theta}}(\re^n)$ of the form
\begin{equation}\label{specialform}
\sigma(x,\xi,\eta)=\sigma_0(x,\xi-\tan(\theta)\,\eta),
\end{equation}
where  $\sigma_0\in S^m_{1,\delta}(\re^n)$ or $\dot{S^m_{1,\delta}}(\re^n),$ respectively.

For a number of  properties of the H\"ormander classes $BS^{m}_{\rho,\delta}(\re^n)$, including symbolic calculus and boundedness properties of the associated bilinear operators with indices related by H\"older scaling, see \cite{BMNT,BBMNT,BNT} and references therein.   The classes $BS^{m}_{\rho,\delta;\theta}(\re^n)$  were first introduced in \cite{BNT} inspired by their $x$-independent versions which originated in  work  on the bilinear Hilbert transform in \cite{LT97} and were extensively studied in \cite{GN02, Ber, BerT} and references therein.

In this section we prove boundedness properties on Lebesgue spaces for bilinear pseudodifferential operator with symbols of negative order where the indices relation is now dictated by the Sobolev scaling. More precisely, the main result in this section is the following:

\begin{theor}\label{thm:pdobounds}
Suppose $n\in\na$  and consider exponents $p_1,p_2\in(1,\infty)$  and $q,\,s>0$ such that
\begin{equation}
 \frac{1}{q}=\frac{1}{p_1}+\frac{1}{p_2}-\frac{s}{n}. \label{eq:po}
\end{equation}
\begin{enumerate}[(a)]
\item  If  $s\in(0,2n),$ $0\le \delta \leq 1,$ and  $\sigma \in BS^{-s}_{1,\delta}(\re^n)\cup \dot{BS^{-s}_{1,\delta}}(\re^n)$ then
$T_\sigma$ is bounded from $L^{p_1}_{w_1} \times L^{p_2}_{w_2}$ into $L^q_w$ for every  pair of weights $(w_1,w_2)$ satisfying the $A_{(p_1,p_2),q}$ condition and $ w:= w_1^{q/p_1}w_2^{q/p_2}.$
\item  If $s\in (0,n),$ $\theta\in (0,\pi)\setminus \{\pi/2,3\pi/4\}$, $0\leq \delta \leq 1$ and $\sigma \in BS^{-s}_{1,\delta;\theta}(\re^n)\cup \dot{ BS^{-s}_{1,\delta;\theta}}(\re^n)$ is of the form \eqref{specialform}
then the bilinear operator $T_\sigma$ is bounded from $L^{p_1} \times L^{p_2}$ into $L^q.$ If in addition $\frac{1}{p}:=\frac{1}{p_1}+\frac{1}{p_2}<1,$ then  $T_\sigma$ is bounded from $L^{p_1} _{w_1}\times L^{p_2}_{w_2}$ into $L^q_{w}$ for weights $w_1,\,w_2$ in the class $A_{p,q}$ and $ w:= w_1^{q/p_1}w_2^{q/p_2}.$
\end{enumerate}
\end{theor}

\begin{proof}
We start with the proof of part (a).
Let  $0\le \delta<1,$ $s\in(0,2n),$ and   $\sigma \in BS^{-s}_{1,\delta}(\re^n)\cup\dot{BS^{-s}_{1,\delta}}(\re^n).$ The results will follow from part \eqref{thm:Ialpha} of Theorem~\ref{fracint} once we have proved that the operator $T_\sigma$ is controlled by the bilinear fractional integral  $\mathcal{I}_s$ as defined in \eqref{def:multfracop}. $T_\sigma$ is given by the spatial representation
$$ T_\sigma(f,g)(x) = \iint_{\re^n \times \re^n} k(x,x-y,x-z) f(y) g(z) dydz$$
where the kernel $k$ is defined by
$$ k(x,u,v):= \widehat{\sigma(x,\cdot,\cdot)}(u,v).$$
 We will prove that,
\begin{equation} \left|k(x,u,v)\right| \lesssim \frac{1}{\left( |u|+|v| \right)^{2n-s}}, \qquad \text{ uniformly in } x, \label{eq:amon} \end{equation}
which gives
$$ \left|T_\sigma(f,g)(x)\right| \lesssim \iint_{\re^n \times \re^n} \frac{|f(y)| |g(z)|}{\left( |x-y|+|x-z| \right)^{2n-s}} dydz = \mathcal{I}_s(|f|,|g|)(x),
$$  and therefore the boundedness properties of $T_\sigma$ follow from  part \eqref{thm:Ialpha} of Theorem \ref{fracint}.

 Let $\Psi(\xi,\eta)$ be a smooth function in $\re^{2n}$ supported on the annulus  $1\le |(\xi,\eta)|\le 2,$  and such that
$$ \int_0^\infty \Psi(t \xi,t \eta) \frac{dt}{t}=1,\quad (\xi,\eta)\neq(0,0).$$
So for each scale $t>0$, we have to estimate $\widehat{\Psi(t\cdot) \sigma(x,\cdot)}$. Now, integration by parts and the hypothesis $\sigma \in BS^{-s}_{1,\delta}(\re^n)\cup\dot{BS^{-s}_{1,\delta}}(\re^n)$  yield
\begin{equation}\label{sizeConv}
\left| \widehat{\Psi(t\cdot) \sigma(x,\cdot)} (u,v) \right| \lesssim \frac{t^{-2n+s}}{\left(1+t^{-1}|(u,v)|\right)^N} 
\end{equation}
for  every large enough integer $N$. Indeed, suppose that $|v| \leq |u| \sim u_j$, so that $|(u,v)| \sim |u| \sim u_j$, then
\begin{align*}
\widehat{\Psi(t\cdot) \sigma(x,\cdot)} (u,v) & = \iint_{\re^n \times \re^n} \Psi(t\xi, t \eta) \sigma(x,\xi,\eta) e^{-i (u \cdot \xi + v \cdot \eta)} d \xi d \eta\\
 & =  \iint_{\re^n \times \re^n} \Psi(t\xi, t \eta) \sigma(x,\xi,\eta)  \frac{1}{(-i)^N u_j^N} \partial^N_{\xi_j}e^{-i (u \cdot \xi + v \cdot \eta)} d \xi d \eta \\
  & =   \frac{1}{(-i)^N u_j^N}  \iint_{|\xi| +|\eta| \sim t^{-1}} \partial^N_{\xi_j} (\Psi(t\xi, t \eta) \sigma(x,\xi,\eta)) e^{-i (u \cdot \xi + v \cdot \eta)} d \xi d \eta.
\end{align*}
But, by the usual Leibniz rule and using the condition on the support of $\Psi$ (which implies $t^{-1} \sim |\xi| + |\eta| \leq 1 + |\xi| + |\eta|$), we have
\begin{align*}
& |\partial^N_{\xi_j} (\Psi(t\xi, t \eta) \sigma(x,\xi,\eta))|  = |\sum_{k=0}^N C_{N,k} \partial_{\xi_j}^{N-k} \Psi(t\xi, t \eta) \partial_{\xi_j}^k \sigma(x,\xi,\eta)|\\
& \leq \sum_{k=0}^N C_{N,k} t^{N-k} |(\partial_{\xi_j}^{N-k}\Psi)(t\xi, t\eta)| \norm{\sigma}{0,k,0} (1+|\xi|+|\eta|)^{-s - k}\\
& \leq  \left(\sup\limits_{0 \leq k \leq N} \norm{\partial^k \Psi}{L^\infty}\right)  \left(\sup\limits_{0 \leq k \leq N} \norm{\sigma}{0,k,0}\right)  \sum_{k=0}^N C_{N,k} t^{N-k}t^{s + k} =: C_{\sigma, N} t^{N+s}.
\end{align*}
Consequently,
\begin{equation}\label{sizeConv1}
|\widehat{\Psi(t\cdot) \sigma(x,\cdot)} (u,v) | \lesssim \frac{t^{N+s}}{u_j^N}  \iint_{|\xi| +|\eta| \sim t^{-1}}   d \xi d \eta \sim \frac{t^{-2n+s}}{(t^{-1}|(u,v)|)^N}.
\end{equation}
On the other hand, again by the hypothesis $\sigma \in BS^{-s}_{1,\delta}(\re^n)\cup\dot{BS^{-s}_{1,\delta}}(\re^n)$, we have
\begin{align*}
& |\widehat{\Psi(t\cdot) \sigma(x,\cdot)} (u,v) | \leq \iint_{\re^n \times \re^n} |\Psi(t\xi, t \eta)| |\sigma(x,\xi,\eta)|  d \xi d \eta\\
& \leq \norm{\Psi}{L^\infty} \norm{\sigma}{0,0,0} \iint_{|\xi| +|\eta| \sim t^{-1}} (1+ |\xi| + |\eta|)^{-s}  d \xi d \eta \lesssim t^{-2n +s},
\end{align*}
and \eqref{sizeConv} follows from this last inequality and \eqref{sizeConv1}. Then, \eqref{sizeConv} and integration over $t\in(0,\infty)$ yield
\begin{align*}
 \left|k(x,u,v) \right| & \lesssim \int_0^\infty \left| \widehat{\Psi(t\cdot) \sigma(x,\cdot)} (u,v) \right| \frac{dt}{t}
   \lesssim \int_0^\infty \frac{t^{-2n+s}}{\left(1+t^{-1}|(u,v)|\right)^N} \frac{dt}{t} \\
  & \lesssim |(u,v)|^{-2n+s} \int_0^\infty \frac{t^{2n-s}}{\left(1+t\right)^N} \frac{dt}{t}  \lesssim |(u,v)|^{-2n+s},
\end{align*}
 which ends the proof of \eqref{eq:amon}.

We now turn to the proof of part (b) of the theorem. If $s\in(0,n),$ $\theta\in (0,\pi)\setminus \{\pi/2,3\pi/4\}$  and  $\sigma \in BS^{-s}_{1,\delta;\theta}(\re^n)\cup \dot{BS^{-s}_{1,\delta;\theta}}(\re^n)$ is of the form $\sigma(x,\xi,\eta)=\sigma_0(x,\xi-\tan(\theta)\,\eta)$ with $\sigma_0\in S^{-s}_{1,\delta}(\re^n)$ or $\sigma_0\in \dot{S^{-s}_{1,\delta}}(\re^n)$ as appropriate, we consider the following spatial representation
for  $T_\sigma:$
$$ T_\sigma(f,g)(x) = \int_{ \re^n} k(x,y) f(x+y) g(x-\tan(\theta)\,y) dy$$
where the kernel $k$ is defined by
$$ k(x,y):= \widehat{\sigma_0(x,\cdot)}(y).$$
Following the same reasoning as above, we obtain
\[
\abs{k(x,y)}\lesssim |y|^{s-n},\quad \text{uniformly in } x,
\]
and therefore
$$ \left| T_\sigma(f,g) \right| \lesssim {\mathcal B}_s(f,g),$$
with ${\mathcal  B}_s$ defined in \eqref{def:B}. The  result then follows from parts  \eqref{thm:Balphaunweighted}  and \eqref{thm:Balphaweighted} of Theorem~\ref{fracint}.
\end{proof}

\begin{remark}  We note that pointwise decay properties of the kernels (and their derivatives) of pseudodifferential operators with symbols in the H\"ormander classes have been studied in \cite[Theorem 5.1]{BMNT}.   In particular, it is proved there that  if $\sigma\in BS^{-s}_{1,\delta}(\re^n),$ then \eqref{eq:amon} holds.
\end{remark}
\begin{remark}\label{numberderivatives} We observe that the  proof of Theorem~\ref{thm:pdobounds} uses the fact that the symbol $\sigma$ satisfies  conditions \eqref{BS1}, \eqref{BS1theta}, or their homogenous counterparts, only for a certain number of derivatives $c_n$ depending only on the dimension $n.$
\end{remark}

\section{Leibniz-type rules in Sobolev spaces}\label{sec:leibnizsob}

In the following, we consider the inhomogeneous and homogeneous Sobolev spaces for indices $s>0$ and $0<p<\infty$,
 \[W^{s,p}(\re^n)=\{f\in\mathcal{S'}(\re^n): J^sf\in L^{p}(\re^n) \}
 \]
 and
 \[
  \dot{W}^{s,p}(\re^n)=\{f\in\mathcal{S'}(\re^n): D^sf\in L^{p}(\re^n) \},
 \]
 where $\mathcal{F}^{-1}$ denotes the inverse Fourier transform, $J^s$ is the operator with Fourier multiplier $(1+|\xi|^2)^\frac{s}{2},$ and $D^s$ is the operator with Fourier multiplier $|\xi|^{s}$. We use the notation $\|f\|_{W^{s,p}}:=\|J^sf\|_{L^p}$ and  $\|f\|_{\dot{W}^{s,p}}:=\|D^sf\|_{L^p}.$
\begin{coro}[Leibniz-type rules] \label{leibniz} Let $n\in\na$ and consider exponents $p_1,p_2\in(1,\infty)$  and $q,\,s>0$ such that
$ \frac{1}{q}=\frac{1}{p_1}+\frac{1}{p_2}-\frac{s}{n}.$
\begin{enumerate}[(a)]
\item  \label{part1}If $s\in(0,2n),$ $0\le \delta<1,$ and  $\sigma \in BS^{m}_{1,\delta}(\re^n)$ for some $m\geq -s$
then
\begin{equation*}
\|T_\sigma(f,g)\|_{L^q}\lesssim \|f\|_{W^{m+s,p_1}}\|g\|_{L^{p_2}}+\|f\|_{L^{p_1}}\|g\|_{W^{m+s,p_2}}.
\end{equation*}

\item \label{part2} If $n\in \na,$ $s\in(0,2n),$ $0\le \delta<1,$ and  $\sigma \in \dot{BS}^{m}_{1,\delta}(\re^n)$ for some $m\geq -s$
then
\begin{equation*}
\|T_\sigma(f,g)\|_{L^q}\lesssim \|f\|_{\dot{W}^{m+s,p_1}}\|g\|_{L^{p_2}}+\|f\|_{L^{p_1}}\|g\|_{\dot{W}^{m+s,p_2}}.
\end{equation*}

\end{enumerate}
\end{coro}

\begin{proof} Part  \eqref{part1} of Corollary \ref{leibniz} follows from Theorem \ref{thm:pdobounds} and composition with  $J^{m+s}$, along the lines  of \cite[Theorem 2]{BNT} (see also \cite[Theorem 1.4]{GK} and \cite[Corollary 8]{BMNT}). Indeed, let $\sigma\in BS^m_{1,\delta}(\re^n)$  and consider $\phi\in C^\infty(\re)$  such that $0 \leq \phi \leq 1$, $\text{supp}(\phi) \subset [-2, 2]$ and $\phi(r) + \phi(1/r) = 1$ on $[0,\infty)$, then, the symbols $\sigma_1$ and $\sigma_2$ defined by
$$
\sigma_1(x,\xi, \eta)= \sigma(x,\xi, \eta) \phi\left(\frac{1+|\xi|^2}{1+|\eta|^2} \right) (1+|\eta|^2)^{-(m+s)/2}
$$
and
$$
\sigma_2(x,\xi, \eta)= \sigma(x,\xi, \eta) \phi\left(\frac{1+|\eta|^2}{1+|\xi|^2} \right) (1+|\xi|^2)^{-(m+s)/2}
$$
are symbols in the class $BS^{-s}_{1,\delta}(\re^n)$, and the operators $T_\sigma$, $T_{\sigma_1}$, and $T_{\sigma_2}$ are related through
$$
T_\sigma(f,g) = T_{\sigma_1}(J^{m+s} f, g) + T_{\sigma_2}(f,J^{m+s} g).
$$

Part \eqref{part2} of Corollary \ref{leibniz} follows in the same way using the operators $D^{m+s}$ instead of $J^{m+s}.$
\end{proof}

We end this section by   presenting  particular cases related to
Theorem~\ref{thm:pdobounds} and Corollary~\ref{leibniz}.

\begin{enumerate}[\tiny$\bullet$]
\item {\bf Fractional Leibniz rule under Sobolev scaling.}

\begin{coro}\label{coro:leibniz} Let $n\in\na,$ $s\in[0,2n),$ $p_1,p_2\in(1,\infty),$ $q>0$
such that $
\frac{1}{q}=\frac{1}{p_1}+\frac{1}{p_2}-\frac{s}{n},$ $m\ge 0$
if $q\ge 1$ and $m>max (0,n-s)$ if $q<1$. Then for functions
defined on $\re^n,$
\begin{equation*}
\|fg\|_{W^{m,q}}\lesssim \|f\|_{W^{m+s,p_1}}\|g\|_{L^{p_2}}+\|f\|_{L^{p_1}}\|g\|_{W^{m+s,p_2}}.
\end{equation*}
\end{coro}
\begin{proof} The case $q\ge 1$ of the above inequality follows from the
Sobolev imbedding $W^{m,q}\subset W^{m+s,r},$
$\frac{1}{r}=\frac{1}{p_1}+\frac{1}{p_2},$ and the well-known
fractional Leibniz rule  \eqref{katoponce}. For the case $q<1$
we proceed as follows:

 Consider $\phi,\,\tilde{\phi}\in C^\infty(\re)$ such that $0 \leq \phi \leq
 1$, $\text{supp}(\phi) \subset [0, \frac{1}{2}],$
 $\text{supp}(\tilde\phi)\subset[\frac{1}{4},4]$ and $\phi(r)
 + \phi(1/r) +\tilde \phi(r) = 1$ on $[0,\infty)$. Then, since $J^m(fg)$ is a bilinear pseudodifferential operator with symbol $(1+|\xi+\eta|^2)^{m/2},$ we get
 $$
J^m(fg) = T_{\sigma_1}(J^{m+s} f, g) + T_{\sigma_2}(f,J^{m+s} g)+T_{\sigma_3}(f,g),
$$
where
\begin{align*}
&\sigma_1(x,\xi, \eta):=  \left(1+|\xi+\eta|^2\right)^{m/2} \phi\left(\frac{1+|\xi|^2}{1+|\eta|^2} \right) (1+|\eta|^2)^{-(m+s)/2},\\
&\sigma_2(x,\xi, \eta):= \left(1+|\xi+\eta|^2\right)^{m/2} \phi\left(\frac{1+|\eta|^2}{1+|\xi|^2} \right)(1+|\xi|^2)^{-(m+s)/2},\\
 &\sigma_3(x,\xi, \eta):=  \left(1+|\xi+\eta|^2\right)^{m/2} \tilde\phi\left(\frac{1+|\xi|^2}{1+|\eta|^2} \right).
\end{align*}
The symbols $\sigma_1$ and $\sigma_2$ belong to the class $BS^{-s}_{1,0}$ ($1+|\xi+\eta|\sim 1+|\eta|$ and $1+|\xi+\eta|\sim 1+|\xi|,$ in the respective supports) and therefore Corollary~\ref{leibniz} imply
that
\[
\|fg\|_{W^{m,q}}\lesssim \|f\|_{W^{m+s,p_1}} \|g\|_{L^{p_2}}+\|f\|_{L^{p_1}}\|g\|_{W^{m+s,p_2}}+\|T_{\sigma_3}(f,g)\|_{L^q}
\]
 Since $1+|\xi+\eta|$ is not comparable to $1+|\eta|$ or $1+|\xi|$ in the support of $\sigma_3$, we cannot expect to prove that this symbol belongs to a suitable  class. We will then  split $\sigma_3$ into elementary symbols. Choose smooth cut-off functions $(\zeta^j)_{1\leq j\leq 3}$, such that $\widehat{\zeta^j}$ is supported on $B(0,4)\setminus B(0,1)$ and
\begin{align*}
\sigma_3(x,\xi,\eta) &= \sum_{l\geq 0} \sum_{l\geq k} 2^{km} \widehat{\zeta^3}\left(\frac{1+|\xi+\eta|^2}{2^{2k}} \right) \widehat{\zeta^1}\left(\frac{1+|\xi|^2}{2^{2l}} \right) \widehat{\zeta^2}\left(\frac{1+|\eta|^2}{2^{2l}} \right)\\
& =: \sum_{l\geq 0} \sum_{l\geq k} m_{k,l}(\xi,\eta).
\end{align*}
Now choose $\Psi^1,\Psi^2$ smooth functions verifying the same
support properties as the $\zeta^j$'s with
$\widehat{\Psi^j}\equiv 1$ on the support of
$\widehat{\zeta^j}$, so that
$$
T_{\sigma_3}(f,g) = \sum_{l\geq 0} \sum_{l\geq k} T_{m_{k,l}}(\Psi^1_{l}(f),\Psi^2_l(g)),
$$
where $\Psi_l$ stands for the usual dilation of $\Psi$ and we
identify $\Psi_l$ with the multiplier it produces. Now we focus
on $K_{k,l}$, the bilinear kernel of $T_{m_{k,l}}$, that is
$$
T_{m_{k,l}}(\Psi^1_{l}(f),\Psi^2_l(g))(x) = \int K_{k,l}(x-y,x-z) \Psi^1_{l}(f)(y) \Psi^2_l(g)(z) dydz.
$$
Then,
$$
\left|K_{k,l}(x-y,x-z)\right| \leq \left|\int e^{i((x-y)\xi+(x-z)\eta)} m_{k,l}(\xi,\eta) d\xi d\eta \right|.
$$
First we notice that $m_{k,l}$ is supported on the set
$\{(\xi,\eta),\ |\xi| \simeq |\eta| \simeq 2^l,\
|\xi+\eta|\simeq 2^{k}\}$ whose measure is bounded by
$2^{n(k+l)}$. After the change of variables $u:=(\xi+\eta)$ and
$v:=(\xi-\eta)$ we get
$$
\left|K_{k,l}(x-y,x-z)\right| \lesssim \left|\int e^{i((2x-y-z)u+(z-y)v)} m_{k,l}\left(\frac{u+v}{2},\frac{u-v}{2}\right) du dv \right|.
$$
Next, integration by parts and the bounds
$$ \left|\partial_u^\alpha \partial_v^\beta m_{k,l}\left(\frac{u+v}{2},\frac{u-v}{2}\right)\right| \lesssim 2^{km} 2^{-k|\alpha|} 2^{-l|\beta|},
$$
yield the following pointwise estimates for $K_{k,l}$
$$
|K_{k,l}(x-y,x-z)| \lesssim 2^{km} \frac{2^{n(k+l)}}{(1+2^k|2x-y-z|+2^l|z-y|)^{2n-s}}.
$$
By Lemma \ref{lem} below, with $m>n-s$, we deduce that
$$
\sum_{k=0}^l |K_{k,l}(x-y,x-z)| \lesssim 2^{lm} \frac{2^{2nl}}{(2^l|2x-y-z|+2^l|z-y|)^{2n-s}}.
$$
Consequently, since $|2x-y-z|+|z-y| \simeq |x-y|+|x-z|$, we get
\begin{align*}
 & |T_{\sigma_3}(f,g)(x)| \leq  \sum_{l\geq 0} \sum_{k=0}^l \left|T_{m_{k,l}}(\Psi^1_{l}(f),\Psi^2_{l}(g))(x)\right| \\
  & \lesssim \sum_{l\geq 0} \iint 2^{lm} \frac{2^{2nl}}{(2^l|y+z-2x|+2^l|y-z|)^{2n-s}} |\Psi^1_l(f)(y) \Psi^2_l(g)(z)| dy dz \\
 & \simeq \sum_{l\geq 0} \iint \frac{1}{(|y+z-2x|+|y-z|)^{2n-s}} 2^{l(m+s)}|\Psi^1_l(f)(y) \Psi^2_l(g)(z)| dy dz \\
 &\simeq \iint \frac{1}{(|y-x|+|z-x|)^{2n-s}} \sum_{l\geq 0}  2^{l(m+s)}|\Psi^1_l(f)(y) \Psi^2_l(g)(z)| dy dz \\
 & \leq \iint \frac{1}{(|y-x|+|z-x|)^{2n-s}} \left(\sum_l 2^{2l(m+s)}|\Psi^1_l(f)(y)|^2\right)^{\frac{1}{2}} \left(\sum_l |\Psi^2_l(g)(z)|^2 \right)^{\frac{1}{2}}\\
 & \simeq {\mathcal I}_s\left( \left(\sum_l 2^{2l(m+s)}|\Psi^1_l(f)|^2\right)^{\frac{1}{2}}, \left(\sum_l |\Psi^2_l(g)|^2 \right)^{\frac{1}{2}} \right)(x).
\end{align*}
 Then the proof follows from the boundedness of the bilinear
 operator ${\mathcal I}_s$ and Littlewood-Paley
 characterizations of Lebesgue spaces, since
 $p_1,p_2\in(1,\infty)$.
\end{proof}

\begin{lemma}\label{lem} For $l \in \na_0$, $a,b,s>0$ and $m, n \in \na_0$ with $m> n-s$, we have
$$
\sum\limits_{k=0}^l \frac{2^{k(m+n)}}{(a2^k + b)^{2n-s}} \lesssim \frac{2^{l(m+n)}}{(a2^l + b)^{2n-s}},
$$
where the implicit constants depend only on $n,m,$ and $s$.
\end{lemma}

\begin{proof} Given $a > 0$, let $k_0 \in \ent$ such that $2^{k_0 -1} \leq a \leq 2^{k_0}$. Suppose first that $0 < b \leq 2^{k_0 +l}$ and write
\begin{align*}
& \sum\limits_{k=0}^l \frac{2^{k(m+n)}}{(a2^k + b)^{2n-s}}\simeq \sum\limits_{k=0}^l \frac{2^{k(m+n)}}{(2^{k+k_0} + b)^{2n-s}} \simeq \sum\limits_{k=k_0}^{l+k_0} \frac{2^{(k-k_0)(m+n)}}{(2^k + b)^{2n-s}} \\
& \simeq \mathop{\sum\limits_{k=k_0}}^{l+k_0}_{2^k \leq b} \frac{2^{(k-k_0)(m+n)}}{(2^k + b)^{2n-s}} + \mathop{\sum\limits_{k=k_0}}^{l+k_0}_{2^k > b} \frac{2^{(k-k_0)(m+n)}}{(2^k + b)^{2n-s}} \lesssim \mathop{\sum\limits_{k=k_0}}^{l+k_0}_{2^k \leq b} \frac{2^{(k-k_0)(m+n)}}{ b^{2n-s}} + \mathop{\sum\limits_{k=k_0}}^{l+k_0}_{2^k > b} \frac{2^{(k-k_0)(m+n)}}{2^{k(2n-s)}}\\
& \lesssim \frac{b^{m+n}}{b^{2n-s}} 2^{-k_0(m+n)} + 2^{-k_0(m+n)} 2^{(k_0+l)[m+n - (2n-s)]}\\
 & \lesssim \left( \frac{b}{2^{k_0}} \right)^{m+n-(2n-s)} 2^{-k_0(2n-s)}+  2^{-k_0(2n-s)} 2^{l [m+n - (2n-s)]} \lesssim  2^{-k_0(2n-s)} 2^{l [m+n - (2n-s)]}\\
& \simeq 2^{l(m+n)} \frac{1}{2^{(l+k_0)(2n-s)}} \simeq   \frac{2^{l(m+n)}}{(2^{l+k_0} + b )^{2n-s}}  \simeq   \frac{2^{l(m+n)}}{(a 2^{l} + b )^{2n-s}}.
\end{align*}
In the case  $b > 2^{k_0 +l}$ we do
\begin{align*}
& \sum\limits_{k=0}^l \frac{2^{k(m+n)}}{(a2^k + b)^{2n-s}}\simeq \sum\limits_{k=0}^l \frac{2^{k(m+n)}}{(2^{k+k_0} + b)^{2n-s}} \simeq \sum\limits_{k=k_0}^{l+k_0} \frac{2^{(k-k_0)(m+n)}}{(2^k + b)^{2n-s}}\\
& \lesssim \frac{1}{ b^{2n-s}} \sum\limits_{k=k_0}^{l+k_0} 2^{(k-k_0)(m+n)} \simeq  \frac{2^{l(m+n)}}{ b^{2n-s}} \simeq \frac{2^{l(m+n)}}{ (2^{k_0+l} + b)^{2n-s}} \simeq \frac{2^{l(m+n)}}{ (a2^{l} + b)^{2n-s}}.
\end{align*}

\end{proof}

\begin{remark} A shorter proof of  Corollary~\ref{coro:leibniz} for $q<1$ can be obtained as follows if we assume $m>c_n$ where $c_n$ is as in Remark~\ref{numberderivatives} (note that $c_n>n-s$).
Consider $\phi\in C^\infty(\re)$  such that $0 \leq \phi \leq 1$, $\text{supp}(\phi) \subset [-2, 2]$ and $\phi(r) + \phi(1/r) = 1$ on $[0,\infty)$
and write
$$
J^m(fg) = T_{\sigma_1}(J^{m+s} f, g) + T_{\sigma_2}(f,J^{m+s} g),
$$
where
$$
\sigma_1(\xi, \eta)= (1+\abs{\xi+\eta}^2)^{m/2} \phi\left(\frac{1+|\xi|^2}{1+|\eta|^2} \right) (1+|\eta|^2)^{-(m+s)/2}
$$
and
$$
\sigma_2(\xi, \eta)= (1+\abs{\xi+\eta}^2)^{m/2} \phi\left(\frac{1+|\eta|^2}{1+|\xi|^2} \right) (1+|\xi|^2)^{-(m+s)/2}.
$$
By Remark \ref{numberderivatives} we can use Theorem~\ref{thm:pdobounds} and conclude that  $T_{\sigma_1}$ and $T_{\sigma_2}$ are bounded from $L^{p_1}\times L^{p_2}$ into $L^q$ if $m> c_n$  and therefore
\begin{align}\label{mlarge}
\|T_{\sigma_1}(J^{m+s}f,g)\|_{L^q} \lesssim \|f\|_{W^{m+s,p_1}}\|g\|_{L^{p_2}},\quad m> c_n,\\
\|T_{\sigma_2}(J^{m+s}f,g)\|_{L^q} \lesssim \|f\|_{L^{p_1}}\|g\|_{W^{m+s,p_2}},\quad m >  c_n,\nonumber
\end{align}
from which the desired result follows.
\end{remark}

\item {\bf Paraproduct estimates under Sobolev scaling.} Let $n\in \na,$ $s\in(0,2n),$ $p_1,p_2\in(1,\infty)$ and $q>0$ such that
$ \frac{1}{q}=\frac{1}{p_1}+\frac{1}{p_2}-\frac{s}{n}.$
Consider a radial, real-valued function  $\varphi\in\mathcal{S}(\re^n)$ such that $\hat{\varphi}(\xi)=1$ for $\abs{\xi}\le 1$ and $\varphi(\xi)=0$ for  $\abs{\xi}\ge 3/2.$ Let  $\psi$ be given by $\hat{\psi}(\xi)=\hat{\varphi}(\xi/2)-\hat{\varphi}(\xi).$
 For $f\in L^1(\re^n)$ we set
 \[
 S_j(f):=\varphi_j* f\qquad \text{and}\qquad \Delta_j(f):=S_{j+1}(f)-S_j(f),
 \]
 where $\varphi_j(x)=2^{jn}\varphi(2^j x),$ $j\in\ent.$ We also define $\psi_j(x):=2^{jn}\psi(2^jx)$ and note that $\text{supp}(\widehat{\psi_j})\subset\{\xi:2^j\le \abs{\xi}\le 3\,2^j\}.$ For $f,\,g\in\mathcal{S}(\re^n)$ we define the Bony paraproduct   of $f$ and $g$ by
\[
 \Pi(f,g):= \sum_{j\in\ent}\Delta_j(f)S_{j-1}(g).
 \]
Straightforward computations show that
$fg= \Pi(f,g)+\Pi(g,f)+\sum_{m=-1}^1R_m(f,g), $ where
$R_m(f,g)=\sum_{j\in\ent}\Delta_j(f)\Delta_{j+m}(g)$ for $m=-1,0,1.$

The symbol $\sigma$ of the paraproduct $\Pi$ is $x$-independent,
\[
\Pi(f,g)(x)=\int_{\re^n}\int_{\re^n} \sigma(\xi,\eta) \hat{f}(\eta)\hat{g}(\xi) e^{ix(\xi+\eta)}\,d\eta\,d\xi,
\]
is given by
\[
\sigma(\xi,\eta)=\sum_{j\in\ent}\widehat{\psi_j}(\xi)\widehat{\varphi_{j-1}}(\eta),
\]
and belongs to the class $\dot{BS}^{0}_{1,0}.$ As a consequence of Corollary~\ref{leibniz}, we have
\begin{equation*}
\|\Pi(f,g)\|_{L^q}\lesssim \|f\|_{\dot{W}^{s,p_1}}\|g\|_{L^{p_2}}+\|f\|_{L^{p_1}}\|g\|_{\dot{W}^{s,p_2}}.
\end{equation*}

\item {\bf Lowering the exponents for linear embeddings.} It is well-known that in $\re^n$, for $s\in (0,1)$, $W^{s,p}$ is
continuously embedded into $L^q$ as soon as $p< d/s$ and $q\geq 1$ with$$ \frac{1}{q}=\frac{1}{p}-\frac{s}{d}.$$

By the previous approach, we get bilinear analogs: indeed we have proved
that $(f,g)\rightarrow fg$ is continuous from $W^{s,p_1} \times W^{s,p_2}$
into $L^q$ as soon as $p<d/s$ (where $p$ is the harmonic mean value of
$p_1,p_2$) and $q>1/2$.
It is then possible to use this bilinear approach to give extensions of
the linear inequalities for $q<1$.

\begin{prop} Let consider $s\in(0,1)$ and $p=t/2< d/s$ and $q\leq 1$ with
$$ \frac{1}{q}=\frac{1}{p}-\frac{s}{d}=\frac{2}{t}-\frac{s}{d}.$$
Then for every nonnegative smooth function $h$, we have
$$ \|h \|_{L^q} \lesssim \| h^{1/2} \|_{W^{s,t}}^2 =
\|h^{1/2}\|_{W^{s,2p}}^2.$$
\end{prop}

\begin{proof} We just write $h =h^{1/2} h^{1/2}$ and apply the bilinear
inequalities to the functions $f=g=h^{1/2}$ with the exponents
$p_1=p_2=t$.
\end{proof}

Such inequalities are of interest since they allow for an exponent $q\leq 1$. To do that we have to pay the cost of estimating the regularity of $\sqrt{h}$.

\end{enumerate}

\section{Acknowledgement}

The authors would like to thank the anonymous referee for his/her careful reading of the manuscript and useful corrections.


\begin{thebibliography}{99}

\bibitem{BB}
N. Badr and F. Bernicot,
\newblock {\emph{New Calder\'on-Zygmund decompositions}},
\newblock{Colloq. Math.} \textbf{121} (2010), 153--177.



\bibitem{BJM}
N. Badr, A. Jim\'enez-del-Toro and J. M. Martell,
\newblock {\emph{$L^p$ self-improvement of generalized Poincar\'e inequalities in
spaces of homogeneous type}},
\newblock {J. Funct. Anal.}, \textbf{260} (2011), no. 11, 3147--3188.


\bibitem{BMNT}
\'A. B\'enyi, D. Maldonado, V. Naibo and R.\ H.\ Torres,
\newblock {\emph{On the H\"ormander classes of bilinear pseudodifferential operators},}
\newblock {Integral Equations Operator Theory} \textbf{67} (2010), 341--364.

\bibitem{BBMNT}
\'A. B\'enyi, F. Bernicot, D. Maldonado, V. Naibo and R.\ H.\ Torres,
\newblock {\emph{On the H\"ormander classes of bilinear pseudodifferential operators II},}
\newblock {preprint}.

\bibitem{BNT}\'A. B\'enyi, A. Nahmod and R.\ H.\ Torres,
\newblock {\it Sobolev space estimates and symbolic calculus for bilinear pseudodifferential operators.}
\newblock { J. Geom. Anal.} \textbf{16} (2006), no. 3, 431--453.

\bibitem{Ber}  F. Bernicot,
\newblock{\it Local estimates and global continuities in Lebesgue spaces for bilinear operators},
\newblock{Anal. PDE} {\bf 1} (2008), 1-27.

\bibitem{BerT}  F. Bernicot and R.\ H.\ Torres,
\newblock{\it Sobolev space estimates for a class of bilinear pseudodifferential operators  lacking symbolic calculus},
\newblock{Anal. PDE {\bf 4} (2011), no. 4, 551-571. }

\bibitem{BLU} A. Bonfiglioli, E. Lanconelli, and F. Uguzzoni,
\newblock{Stratified Lie groups and potential theory for their sub-Laplacians,}
\newblock{Springer Monographs in Mathematics,} Springer, Berlin, 2007.





\bibitem{CWein} M. Christ and M. Weinstein,
\newblock {\it Dispersion of small amplitude solutions of the generalized Korteweg-de Vries equation},
\newblock {J. Funct. Anal.} {\bf 100}, (1991), 87--109.

\bibitem{cm}
R. R.  Coifman and Y. Meyer,
\newblock {Au del{\`a} des op{\'e}rateurs pseudo-diff{\'e}rentiels,}
\newblock {Ast\'erisque} \textbf{57}  SMF, 1978.


\bibitem{DY1}
X. Duong and L. Yan,
\newblock {\it Duality of Hardy and BMO spaces associated with operators with heat kernel bounds},
\newblock { J. Amer. Math. Soc.} \textbf{18} (2005), no. 4, 943--973.


\bibitem{DY2}
X. Duong and L.~Yan,
\newblock {\it New function spaces of {BMO} type, the {J}ohn-{N}iremberg inequality,
  {I}nterplation and {A}pplications,}
\newblock {Comm. Pure Appl. Math.} \textbf{58}, no.10 (2005), 1375-1420.


\bibitem{GN02}
J. Gilbert and A. Nahmod,
\newblock{ \it $L^p$-boundedness for time-frequency paraproducts. II.},
\newblock{J. Fourier Anal. Appl.} \textbf{8} (2002), no. 2, 109--172.


\bibitem{G}
L. Grafakos,
\newblock {\it On multilinear fractional integrals,}
\newblock  {Studia Math.} \textbf{102} (1992), 49--56.


\bibitem{GT}
L. Grafakos and R.\ H.\ Torres,
\newblock{\it Multilinear Calder\'{o}n-Zygmund theory,}
\newblock{Adv. in  Math.} {\bf 165} (2002), 124--164.

\bibitem{GK} A. Gulisashvili and M. Kon,
\newblock{\it Exact smoothing properties of Schr\"odinger semigroups,}
\newblock{Amer. J. Math.} \textbf{118}, (1996), 1215--1248.

\bibitem{JM} A. Jim\'enez-del-Toro and J. M. Martell, \emph{Self-improvement of Poincar\'e-type inequalities associated with approximations of the identity and semigroups}, preprint.

\bibitem{KaPo}
T. Kato and G. Ponce,
\newblock {\it Commutator estimates and the Euler and Navier-Stokes equations.}
\newblock {Comm. Pure Appl. Math.} {\bf 41}, (1988),  891--907.

\bibitem{KPVe}
C. Kenig, G. Ponce and L. Vega,
\newblock {\it Well-posedness and scattering results for the generalized Korteweg-de Vries equation via the contraction principle},
\newblock {Comm. Pure Appl. Math.} {\bf 46} (1993), 527--620.

\bibitem{KS}
C.  Kenig and E. M. Stein,
\newblock  {\it Multilinear estimates and fractional integration,}
\newblock {Math.  Res.  Lett.} {\bf  6} (1999), 1--15.


\bibitem{LT97}
M. Lacey and C. Thiele,
\newblock{ \it $L^p$-bounds for the bilinear Hilbert transform, $2<p<\infty$,}
\newblock{Ann. of Math.} (1997) (2) \textbf{146}, 693--724.


\bibitem{LOPTT}
A. Lerner, S. Ombrosi, C. P\'erez, R. H. Torres and R. Trujillo-Gonz\'alez,
\newblock {\it New maximal functions and multiple weights for the multilinear Calder\'on-Zygmund theory, }
\newblock {Adv. in Math.} \textbf{220} (2009), no. 4, 1222--1264.


\bibitem{Lu98}
 G. Lu,
 \newblock {\it Embedding theorems on Campanato-Morrey spaces for vector fields of H\"ormander type},
 \newblock {Approx. Theory Appl.}, \textbf{14} (1), (1998), 69--80.

\bibitem{MMN}
D. Maldonado, K. Moen, and V. Naibo,
\newblock {\it Weighted multilinear Poincar\'e inequalities for vector fields of H\"ormander type,}
\newblock {Indiana Univ. Math. J.}, to appear.


\bibitem{MN09b}
D. Maldonado and V. Naibo,
\newblock{\it On the boundedness of bilinear operators on products of Besov and Lebesgue spaces,}
\newblock{J. Math. Anal. Appl.} \textbf{352} (2009), 591--603.


\bibitem{Mc}
A. McIntosh,
\newblock {\it Operators which have an $H_\infty$-calculus},
\newblock {Miniconference on operator theory and partial differential equations} (1986) Proc. Centre Math. Analysis, ANU, Camberra \textrm{14}, 210--231.


\bibitem{Kabe}
K. Moen,
\newblock {\it Weighted inequalities for multilinear fractional integral operators,}
\newblock { Collect. Math.} \textbf{60}, (2009), 213--238.


\bibitem{MN}
K. Moen and V. Naibo,
\newblock {\it Higher-order multilinear Poincar\'e inequalities in Carnot groups,}
\newblock { preprint}.

\bibitem{MW} B. Muckenhoupt and R. Wheeden,
\newblock{\it Weighted norm inequalities for fractional integrals},
\newblock{Trans. Amer. Math. Soc.} {\bf 192}, (1974), 261--274.





\end{thebibliography}
\end{document}